\documentclass{amsart}
\usepackage{eucal}
\usepackage{amssymb}
\usepackage{amsthm}
\usepackage{booktabs}
\usepackage{graphicx,color}   
\usepackage{hyperref}         
\hypersetup{colorlinks, 
  citecolor=black,
  filecolor=black,
  linkcolor=blue,
  urlcolor=red,
  pdftex}
\newtheorem{theorem}{Theorem}[section]
\newtheorem{lemma}[theorem]{Lemma}
\newtheorem{corollary}[theorem]{Corollary}

\newcommand{\vvv}[1]{{{\mathbf #1}}}  
\newcommand{\VVV}[1]{{{\mathbf{#1}}}} 

\newcommand{\va}{{\vvv a}}

\newcommand{\vc}{{\vvv c}} 
\newcommand{\vg}{{\vvv g}}
\newcommand{\vp}{{\vvv p}} 
\newcommand{\vv}{{\vvv v}}
\newcommand{\vw}{{\vvv w}}
\newcommand{\vx}{{\vvv x}} 
\newcommand{\vC}{{\VVV C}} 
\newcommand{\vP}{{\VVV P}}
\newcommand{\vT}{{\VVV T}}
\newcommand{\vW}{{\VVV W}}
\newcommand{\vX}{{\VVV X}}
\newcommand{\vZ}{{\VVV Z}}

\newcommand{\vGa}{{\Gamma}} 
\newcommand{\ah}{{\hat\va}}
\newcommand{\tvc}{\tilde\vc}

\newcommand{\N}{\mathbb{N}}
\newcommand{\Q}{\mathbb{Q}}
\newcommand{\R}{\mathbb{R}}
\newcommand{\B}{\mathbb{B}}
\newcommand{\Sph}{\mathbb{S}}
\newcommand{\Z}{\mathbb{Z}}

\newcommand{\cA}{\mathcal{A}}
\newcommand{\cB}{\mathcal{B}} 
\newcommand{\cJ}{\mathcal{J}} 
\newcommand{\cM}{\mathcal{M}} 
\newcommand{\cO}{\mathcal{O}} 
\newcommand{\cR}{\mathcal{R}} 
\newcommand{\cS}{\mathcal{S}}

\newcommand{\fG}{\mathfrak{G}}

\newcommand{\eps}{\epsilon}  
\newcommand{\identity}{{\mathbf 1}}
\newcommand{\pro}{\mathfrak{p}}

\newcommand{\Rex}{{\cR_{\rm ex}}}
\newcommand{\Rin}{{\cR_{\rm in}}}

\newcommand{\fl}{}

\begin{document}

\title {The Zoo of Solitons for Curve Shortening in $\R^n$}
\author[Altschuler, Altschuler, Angenent \& Wu]{Dylan~J.~Altschuler,
  Steven~J.~Altschuler,\\
  Sigurd~B.~Angenent and Lani~F.~Wu}

\address{DJA: St.~Mark’s School
  of Texas, 10600 Preston Rd, Dallas, Texas}

\address{SJA,LWu: Green
  Center for Systems Biology, UT Southwestern Medical Center, Dallas,
  Texas}

\address{SBA: Mathematics Department, University of Wisconsin
  at Madison }

\date{\today}
\maketitle

\begin{center}
  \it On the occasion of Richard Hamilton's {\rm n}th birthday
\end{center}

\begin{abstract}
  We provide a detailed description of solutions of Curve Shortening
  in $\R^n$ that are invariant under some one-parameter symmetry group
  of the equation, paying particular attention to geometric properties
  of the curves, and the asymptotic properties of their ends.  We find
  generalized helices, and a connection with curve shortening on the
  unit sphere $\Sph^{n-1}$.  Expanding rotating solitons turn out to
  be asymptotic to generalized logarithmic spirals.  In terms of
  asymptotic properties of their ends the rotating shrinking solitons
  are most complicated.  We find that almost all of these solitons are
  asymptotic to circles.

  Many of the curve shortening solitons we discuss here are either
  space curves, or evolving space curves.  In addition to the figures
  in this paper, we have prepared a number of animations of the
  solitons, which can be viewed at
  \url{http://www.youtube.com/user/solitons2012/videos?view=1}.
\end{abstract}


\section{Introduction}
\label{sec:introduction}

A family $\vc:(t_0, t_1)\times\R \to \R^n$ of smooth parametrized
curves in $\R^n$ evolves by \emph{Curve Shortening} if it satisfies
\begin{equation}
  \pro_{\vc_s}\bigl(\vc_t\bigr) = \vc_{ss}
  \label{eq:CS}
\end{equation}
where $s$ is arc length along the curve, and $\pro_{\va}$ denotes
orthogonal projection along the vector $\va$, so that
$\pro_{\vc_s}\vc_t$ is the component of the velocity vector $\vc_t$
normal to the curve.  If the curves are parametrized so that
$\vc_t\perp\vc_s$ at all times, then (\ref{eq:CS}) is equivalent with
\[
\vc_t = \vc_{ss}.
\]

Curve Shortening for compact curves has been extensively studied, especially in
the case of curves in the plane or on a surface (see
e.g.~\cite{2001ChouZhuCSBook}.)  Here we look at Curve Shortening for noncompact
curves in $\R^n$.  More specifically, we consider \emph{solitons},
i.e.~\emph{self-similar solutions}.  Since noncompact curves can be
geometrically much more complex than compact curves, it is to be expected that
Curve Shortening for non-compact curves will present a much wider range of
phenomena than for compact curves.

In geometric analysis, solitons are considered to be solutions of some
geometric evolution equation that are invariant under some subgroup of the
full symmetry group of the equation.  The equation (\ref{eq:CS}) is
invariant under Euclidean motions in space, translations in time, and
parabolic dilations of space-time.  If $\fG$ is the group generated by
these transformations of space-time, then we will classify and describe all
solutions of (\ref{eq:CS}) that are invariant under some one-parameter
subgroup of $\fG$.  In \S\ref{sec:thegroupG} we classify the one-parameter
subgroups of $\fG$ up to conjugation and find three essentially different
kinds of one-parameter subgroup and thus three types of self similar
solution of (\ref{eq:CS}) (see Table~\ref{tab:list-of-subgroups}).

\begin{description}
\item[Solutions with static symmetry]
  The first kind of solution we encounter (``category A'' in
  Table~\ref{tab:list-of-subgroups}) consist of curves that at each moment in
  time are invariant under some family of Euclidean motions.  The most general
  static symmetry that is preserved by Curve Shortening is of the form
  \begin{equation}
    \vC\mapsto e^{\theta\cM}\vC + \theta\vv,
    \qquad (\theta\in\R)
    \label{eq:static-symmetry}
  \end{equation}
  for any fixed skew symmetric matrix $\cM$ and any fixed vector $\vv$ with
  $\cM\vv=0$.  
\item[Rotating and translating solitons] The solutions in category B
  (Table~\ref{tab:list-of-subgroups}) are defined for all time $t\in\R$.  They
  are characterized by the fact that the curve $\vc(t, \cdot)$ at any time $t\in\R$ is
  obtained from the initial curve $\vc(0, \cdot)$ by a combination of translation and
  rotation, i.e.
  \[
  \vc(t, \xi) = e^{t\cA} \vc(0, \xi) + t\vv.
  \]
  Here $\vv$ is the translation velocity and $\cA = -\cA^t$ is an antisymmetric
  matrix representing the rate at which the soliton rotates.

\item[Rotating and dilating solitons] Solitons in category C
  (Table~\ref{tab:list-of-subgroups}) are characterized by the
  condition that the curve $\vc(t, \cdot)$ at any time $t$ is obtained
  from the curve $\vc(0)$ at time $t=0$ by a combination of dilation
  and rotation, i.e.
  \[
  \vc(t, \xi) = (1+2\alpha t)^{\frac12 + \frac{1}{2\alpha} \cA}\; \vc(0, \xi),
  \]
  where $\alpha\neq0$ is a constant and $\cA=-\cA^t$ is the skew symmetric
  matrix that describes the rotation.
  For $\alpha=0$ we have
  \[
  \vc(t,\xi) = e^{ t \cA}\; \vc(0, \xi)
  \]
  instead.

\end{description}

Some special well-known solutions of Curve Shortening include
\emph{straight lines} (which fall under all three categories);
\emph{circles} (which have static symmetry, and which also evolve by
dilating), the \emph{Grim Reaper,} i.e.~the graph of $y=-\log\cos x$
(which moves by steady translation upward along the $y$-axis); the
\emph{Yin-Yang curve} (a spiral in the plane which moves by steady
rotation); the \emph{Abresch-Langer curves} (plane curves that evolve
by shrinking dilation\footnote{Abresch and Langer derived and analyzed
  the ODE that describes planar dilating solitons.  A main point of
  their paper \cite{AbreschLanger86} was to show that while most of
  these curves are not compact, there is a discrete family of compact
  immersed curves among them.  Nonetheless, we will call both compact
  and non compact purely shrinking solitons ``Abresch-Langer
  curves.''}); the \emph{Brakke wedge} (a convex plane curve that
evolves by expanding dilation.)  See Figure~\ref{fig:3dsolitonzoo}.
The case of plane solitons has been fully analyzed by Halldorsson in
\cite{2010Halldorsson}, where one can also find many figures of plane
solitons.

\subsection*{Summary of main results}
The existence of special solutions with static symmetry was first reported in
\cite{Altschuler93,AltschulerGrayson92,Altschuler91} where it was noted that
under Curve Shortening a helix in $\R^3$ will remain a helix, and that this
helix will shrink to its central axis.  Since Curve Shortening preserves the
symmetries (\ref{eq:static-symmetry}), any solution of (\ref{eq:CS}) whose
initial curve is invariant under (\ref{eq:static-symmetry}) will at all times be
invariant under (\ref{eq:static-symmetry}).  The symmetries in
(\ref{eq:static-symmetry}) form a one parameter group so that curves invariant
under this group are exactly the orbits of the group action on $\R^n$.  The
Curve Shortening flow on curves invariant under (\ref{eq:static-symmetry}) can
thus be reduced to a finite system of ODEs.  In \S\ref{sec:shrinking-helices} we
derive this system and find that after a time change it is equivalent with the
first order linear system $\vx' = \cM^2 \vx$, which is trivially solved.
Translating the solutions of $\vx'=\cM^2\vx$ back to curves evolving by
(\ref{eq:CS}), we find two somewhat different cases.  For $\vv\neq0$ the
solutions behave as in the Altschuler-Grayson helix example, namely, in backward
time they are helix-like curves wound on a cylinder of radius $\cO(\sqrt{-t})$
while in forward time the curves converge at an exponential rate to the line
through the origin in the direction of $\vv$.  For $\vv=0$, solutions only
exist for $t\in(-\infty, T)$ for some $T$.  At time $t$ the curve is contained
in a sphere of radius $\sqrt{2(T-t)}$.  After rescaling to the unit sphere, the
solutions converge both as $t\to-\infty$ and as $t\to T$ to different great
circles.  In fact, after changing time, the rescaled solutions are solutions to
Curve Shortening for immersed curves on the unit sphere $\Sph^{n-1}$.

Starting in \S\ref{sec:solitons} we consider solitons. Any rotating-translating
or rotating-dilating soliton $\vc(t)$ is completely determined by its initial
curve $\vc(0)$.  A curve $\vC$ is the initial curve of a rotating-dilating or a
rotating-translating soliton exactly when it is the solution of the second order
ODE
\begin{equation}
  \vC_{ss} = \bigl(\alpha + \cA\bigr)\vC + \vv + \lambda \vC_s,
  \label{eq:soliton-profile}
\end{equation}
where $\lambda$ is determined by the condition $\vC_s \perp \vC_{ss}$.
The parameters  $\alpha$, $\cA$ and $\vv$ are as above.

The case of non-rotating solitons, where $\cA=0$, turns out not to present
any new curves in higher dimensions.  On one hand it is a well-known folk
theorem that all purely translating solitons are planar, so that up to
rotation and scaling they are equivalent to the Grim Reaper (we prove a
more general statement in Lemmas \ref{lem:dilating-rotating-is-2D-in-NA},
\ref{lem:translating-rotating-is-2D-in-NA}, and \ref{lem:purely-rotatin}).
On the other hand, it is also known that all purely dilating solitons are
planar.  They are therefore copies of the Abresch-Langer solitons. Again,
Lemma \ref{lem:dilating-rotating-is-2D-in-NA} presents a generalization of
this statement.

The classification in \S\ref{sec:thegroupG} of one-parameter subgroups of the
complete group $\fG$ of symmetries of Curve Shortening shows that we do not have
to consider solitons that simultaneously dilate and translate.  While the
soliton equation (\ref{eq:soliton-profile}) certainly has solutions when
$\alpha\neq0$ and $\vv\neq0$, these curves turn out to be non-translating
solitons whose center of dilation and rotation is not at the origin.  We
therefore only consider rotating-translating or rotating-dilating solitons.

In studying the global behaviour of the curve we distinguish between unbounded
ends of a soliton, and ends that stay within a bounded region.  We show in
\S\ref{sec:distance-to-Origin} that the unbounded ends can be decomposed into
high curvature and low curvature parts.  The high curvature parts are
characterized by the condition that the tangent and the vector $\va =
(\alpha+\cA)\vC+\vv$ are not almost parallel.  These high curvature arcs are
short (their length is $\cO(\|\va\|^{-1})$), and the further away they are from
the origin, the better they are approximated by a small copy of the Grim
Reaper. (See Figure~\ref{fig:GR-in-the-distance}.) On the remaining low
curvature arcs, which comprise most of the soliton, the tangent $\vC_s$ is
almost parallel to $\va$, so that these arcs are solutions of
\begin{equation}
  \lambda \frac{d\vC}{ds} = (\alpha+\cA)\vC + \vv + \mbox{ a small error term}.
  \label{eq:low-curvature-approximate-ode}
\end{equation}
After reparametrization this equation is linear, and easily solved.  It
should provide a good description of ``medium length'' sections of
solitons. This description can in certain cases be shown to be globally
correct (see below, and Lemma~\ref{lem:alpha-pos-R-invariant}.)

Of all the rotating-dilating solitons, the ones that rotate and expand are
simplest to describe.  The distance to the origin from a point on a
rotating-expanding soliton is a function on the curve that has exactly one
global minimum, and no other critical points.  On both ends of the curve
the distance to the origin grows without bound.  Both ends are low
curvature arcs, and using the approximate equation
\ref{eq:low-curvature-approximate-ode} we show that the ends are asymptotic
to generalized logarithmic spirals (Lemma \ref{lem:alpha-pos-R-invariant}).
In the very special case where $\cA=0$ these spirals reduce to straight
lines, and the solitons are just the Brakke wedges.

For purely rotating solitons the approximate equation
(\ref{eq:low-curvature-approximate-ode}) (with $\alpha=0$, $\vv=0$) suggests
that the ends will merely rotate.  However we find in
\S\ref{sec:distance-to-Origin} that the ends must be unbounded, which shows that
the ends of purely rotating solitons slowly ``spiral off to infinity.''  If the
null space $N(\cA)$ is non-empty then either the soliton is contained in a
proper subspace of $\R^n$, or else $N(\cA)$ contains a line $\ell$ such that the
orthogonal projection of the soliton onto $\ell$ is injective
(Lemma~\ref{lem:purely-rotatin}).  This must always be the case in odd
dimensions; e.g.~in $\R^3$ the null space $N(\cA)$ is the axis of the rotation
generated by $\cA$, and Lemma~\ref{lem:purely-rotatin} says that the soliton is
either planar, or a graph over the axis of rotation.

\begin{figure}[b]
  \begin{center}
    \includegraphics[width=\textwidth]{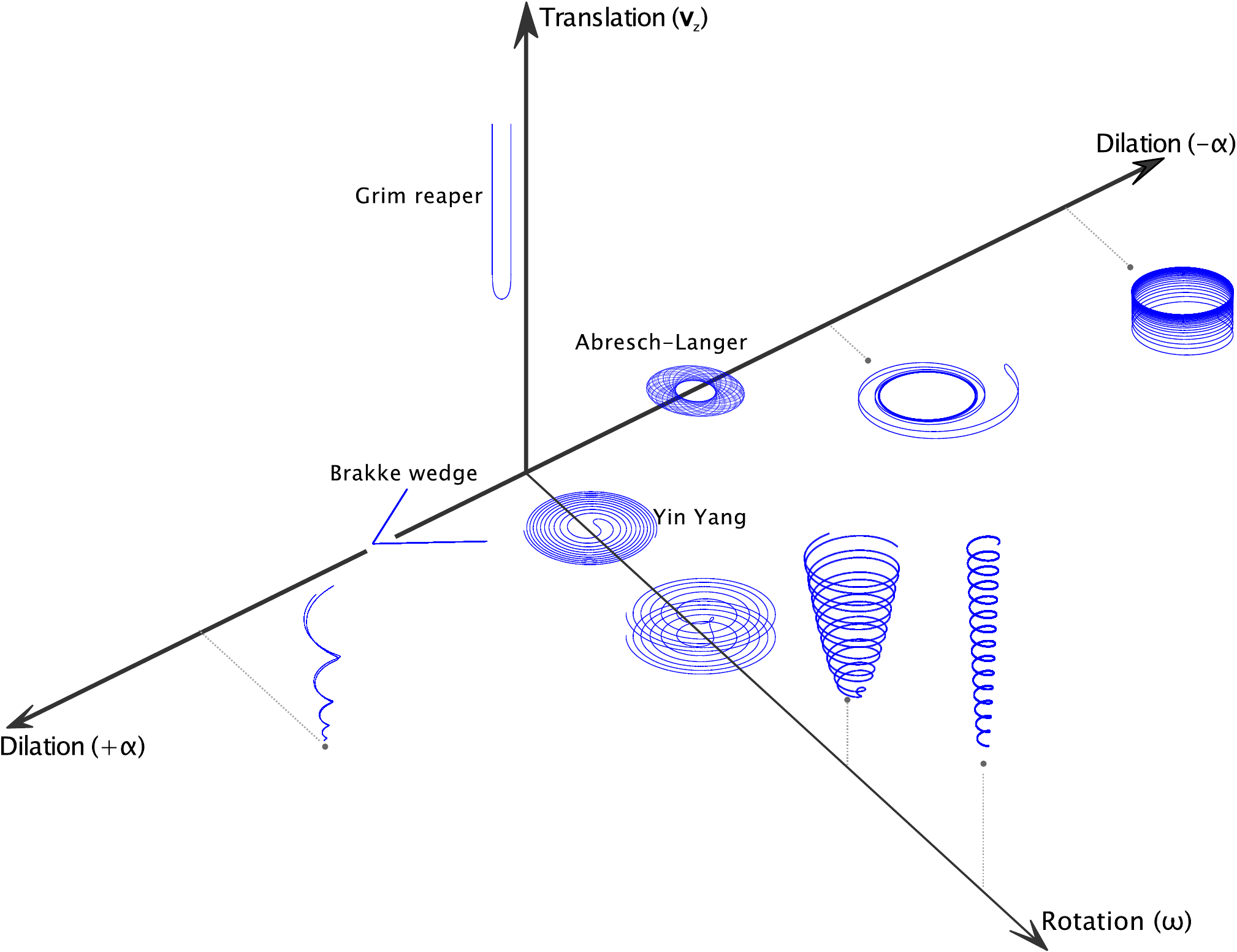}
  \end{center}
  \caption{For three dimensional solitons the soliton equation
  (\ref{eq:soliton-profile}) contains three essential parameters: the dilation
  rate $\alpha$, the rotation rate $\omega$, and the translation velocity
  $v_{z}$. For each fixed choice of $( \alpha, \omega, v_{z})$ many solitons
  exist, namely, one through each point $\vC\in\R^{3}$ and unit tangent vector
  $\vT$. This figure shows a few solitons in $\R^{3}$.  The named solitons are
  well known planar examples.}
  \label{fig:3dsolitonzoo}
\end{figure}

A similar result holds for translating-rotating solitons ($\alpha=0$,
$\cA\vv=0$, $\vv\neq0$.)  Here the projection of the soliton onto its
translation axis (the line through the origin in the direction of $\vv$) is
bijective, or else the function $\vx \mapsto \langle\vv, \vx\rangle$ has a
unique minimum on the soliton.  In the latter case the projection onto the
translation axis is two-to-one except at the point where $\langle\vv,
\vx\rangle$ is minimal (see Figure~\ref{fig:solitons-and-nullspace}.) The
quantity $\langle\vv, \vx\rangle$ is unbounded along each end.  If $\langle\vv,
\vx\rangle\to \infty$ along an end, then along this end the distance to the
translation axis must also become unbounded.  Any end along which $\langle\vv,
\vx\rangle\to - \infty$ converges to the translation axis
(Lemma~\ref{lem:ends-of-translating-rotating}).

The most complicated solitons are the rotating shrinking solitons ($\alpha<0$,
$\cA\neq0$, $\vv=0$).  First, given any $\cA$ any circle of radius
$1/\sqrt{-\alpha}$ that is invariant under the rotations $e^{\theta\cA}$ is
trivially a soliton that dilates and rotates with parameters $(\alpha, \cA)$.
Similarly, any Abresch-Langer curve in $N(\cA)$ is a dilating-rotating soliton
with the same parameters.  In \S\ref{sec:global-behaviour} we show that any
bounded end of a rotating shrinking soliton must converge to these trivial
cases, i.e.~$e^{\theta\cA}$ invariant circles, or Abresch-Langer curves in
$N(\cA)$.

The class of rotating shrinking solitons also includes the rotating
solitons for Curve Shortening of immersed curves on the sphere
$\Sph^{n-1}$.  These were studied by Hungerb\"uhler and Smoczyk in
\cite{2000HungerBuehlerSmoczyk} (in \cite{2000HungerBuehlerSmoczyk}
solitons on other surfaces were also considered).  The connection is
explained in \S\ref{sec:CS-Sn-1} and \S\ref{sec:rotators-on-sphere}.

In Lemmas~\ref{lem:dVdtau} and \ref{lem:V-critical-extended} we show that a
shrinking-rotating soliton can have at most one unbounded end, unless it is a
straight line in $N(\cA)$.  If a rotating-shrinking soliton has an unbounded
end, then we find that there are two distinct cases.  One possibility is that
the distance from points on the end to $N(\cA)$ never falls below a certain
lower bound (specified in Lemma~\ref{lem:alpha-pos-R-invariant}).  In this case
the end is asymptotic to a generalized logarithmic spiral.  In
Lemma~\ref{lem:alpha-neg-R-isolating-block} we prove existence of such solitons.
The other possibility is that the unbounded end of the soliton stays within a
fixed distance of $N(\cA)$.  In this case it must accumulate on Abresch-Langer
curves and straight lines in $N(\cA)$.  Although we were initially inclined to
conjecture that this last case could not occur, preliminary computations
involving matched asymptotic expansions, which we intend to present in a
forthcoming paper, indicate that that such exotic solitons may actually exist.

\bigskip

\textbf{Acknowledgments. }  The second author (SJA) began exploring the world
of 3D curve shortening solitons together with John Sullivan while at the IMA
in 1990.  This present paper owes much of its inspiration to those preliminary 
investigations.

SJA was supported by NIH GM071794 and the Welch Foundation I-1619.
LFW was supported by NIH GM081549 and the Welch Foundation I-1644.
SBA was supported by NSF on grant DMS0705431.


\section{The symmetry group for Curve Shortening and its one-parameter
subgroups}
\label{sec:thegroupG}
The Curve Shortening equation is invariant under translations in space-time,
rotations in space, and, being a heat equation, parabolic dilations of
space-time.  In this section we classify the one-parameter subgroups of the full
symmetry group of Curve Shortening.  We then describe what solutions invariant
under each of these subgroups look like, and derive the differential equations
they satisfy.

Let $\fG$ be the group of transformations of space-time $\R^n\times\R$
of the form
\begin{equation}
  g:\R^n\times\R \to \R^n\times\R, \quad
  g(\vx, t) = \bigl(
  \Theta \cR\vx + \va,\; \Theta^2 t + b
  \bigr)
  \label{eq:general-symmetry}
\end{equation}
where $\Theta>0$, $b\in\R$, $\cR\in O(n,\R)$, and $\va\in\R^n$.  Any
transformation $g\in\fG$ leaves (\ref{eq:CS}) invariant, i.e.\ it maps
solutions of (\ref{eq:CS}) to solutions of (\ref{eq:CS}).

Let
\begin{equation}
  g_\eps(\vx,t)  =
  \bigl(\Theta(\eps)\cR(\eps)\vx+\va(\eps),
  \Theta(\eps)^2t+b(\eps)\bigr)
  \label{eq:one-parameter-sgp}
\end{equation}
be a one-parameter subgroup of $\fG$.  Such a subgroup is completely
determined by the vector field
\[
V(\vx, t) = \left.\frac{\partial g_\eps(\vx, t)} {\partial
\eps}\right|_{\eps=0} =\bigl( ( \theta + \cM) \vx + \vv,\; 2
\theta t + w \bigr),
\]
where $\theta, \vv, w$, and $\cM$ are defined by
\[
\left.\frac{d} {d\eps}\right|_{\eps=0} \bigl(\Theta(\eps), \va(\eps),
b(\eps), \cR(\eps)\bigr) = \bigl( {\theta}, {\vv}, {w}, {\cM}\bigr).
\]
Specifically, the orbit $g_\eps(\vx_0, t_0) = (\vx(\eps), t(\eps))$
through any point $(\vx_0, t_0)$ is the solution of the first-order
differential equation $(\dot \vx, \dot t) = V(\vx, t)$, where the dot
denotes differentiation with respect to $\eps$.  Thus, the general one
parameter subgroup of $\fG$ is obtained by integrating the
differential equations
\begin{equation}
  \dot \vx = (  \theta +   \cM) \vx +   \vv, \qquad
  \dot t = 2 \theta t +   w
  \label{eq:x-t-flow}
\end{equation}
for suitable choice of constants $ \theta \in\R$, $ w \in\R$, $
\vv\in\R^n$ and $ \cM\in\mathfrak{so}(n, \R)$ (i.e.\ $
\cM:\R^n\to\R^n$ is antisymmetric).

Instead of considering the most general one-parameter subgroup generated by
(\ref{eq:x-t-flow}), one can simplify the problem by conjugating the subgroup
with some $h\in\fG$, i.e.~by applying a linear change of coordinates $(\vx, t) =
h(\hat{\vx}, \hat t)$.  Below we choose simple representatives in each conjugacy
class of subgroups.

\subsection*{Choice of parameter for the subgroup}
If $(\theta, \vv, w, \cM)$ generates the subgroup $\{g_\eps\}$, then
the one parameter subgroup $\{h_\eps\}$ generated by any multiple
$(\lambda \theta, \lambda \vv, \lambda w, \lambda \cM)$ is given by
$h_\eps = g_{\lambda \eps}$.  Therefore multiplying the vector field
$V$ with a constant does not change the subgroup it generates.  In
particular, we may always assume that
\begin{equation}
  w\geq 0.
  \label{eq:can-choose-w-nonneg}
\end{equation}
\subsection*{Rotating the coordinates}
We can rotate in space and substitute $\vx = \cS\hat\vx$, $t=\hat t$,
for any orthogonal $\cS\in O(n,\R)$.  Then $\hat t$ and $t$ satisfy
the same equation, while $\hat\vx$ satisfies
\[
d\hat\vx/d\eps = \bigl(\theta + \cS^{-1}\cM\cS\bigr) \hat \vx +
\cS^{-1}\vv.
\]
Since $\cM:\R^n\to\R^n$ is antisymmetric, its rank is even, say
rank$(\cM)=2m$, and there is an $\cM$-invariant orthogonal
decomposition
\begin{equation}
  \R^n=E_1\oplus E_2 \oplus \cdots \oplus E_m \oplus N(\cM)
  \label{eq:Rn-cA-invariant-splitting}
\end{equation}
in which each $E_j$ is two dimensional.  This
allows us to choose the rotation $\cS$ so that after conjugation with
$\cS$ the matrix of $\cM$ has the form
\begin{equation}
  \cM =
  \left(
  \begin{array}{cccc}
    \omega_1\cJ & & 0 & \\
    &\ddots && \\
    0&&\omega_m\cJ& \\
    &&&{\bf 0}_{n-2m}
  \end{array}
  \right)
  \mbox{ with }
  \cJ =
  \left(
  \begin{array}{cc}
    0 & -1 \\ 1 & 0
  \end{array}
  \right)
  \label{eq:rotation-normal-form}
\end{equation}
where the rotation rates $\omega _j$ are positive and ordered
\[
0< \omega_1 \leq \omega_2 \leq \cdots\leq \omega_m,
\]
and where ${\bf 0}_{n-2m}$ is the $(n-2m)\times (n-2m)$ zero matrix.

\subsection*{Translating in space}
After rotating space, we can relocate the origin in space to $\vp$ and
substitute $\vx=\hat{\vx}+\vp$, $t=\hat t$.  The equations
(\ref{eq:x-t-flow}) then become
\[
d\hat\vx/d\eps = ( \theta + \cM) \hat\vx + \hat \vv, \qquad d\hat
t/d\eps = 2 \theta \hat t + w
\]
where
\[
\hat\vv = \vv + ( \theta + \cM) \vp
\]

If $\theta\neq0$ then, since $ \cM$ is antisymmetric, $ \theta + \cM$
is invertible, and so we can choose $\vp = - ( \theta + \cM)^{-1} \vv$
to make $\hat\vv=0$.  In short: if $ \theta\neq0$ then, after choosing
a different origin in space, we may assume that $ \vv=0$.

If $ \theta=0$ then $\hat\vv = \vv + \cM\vp$, so that we can still
choose $\vp$ to make $\hat\vv\perp\mbox{Range}( \cM)$.  Since $\cM$ is
antisymmetric, this is the same as choosing $\vp$ so that
$\cM(\hat\vv) = 0$.

\subsection*{Translating in time}
Having chosen the origin in space, we can translate in time by setting
$t = \vartheta + \hat t$.  This leaves the equation for $\vx$ in
(\ref{eq:x-t-flow}) unchanged and turns the $t$-equation into
\[
d\hat t/d\eps = 2\theta \hat t + w + 2\theta\vartheta = 2\theta\hat t
+ \hat w.
\]
When $ \theta=0$ this has no effect at all, but if $ \theta\neq0$ we
can choose $\vartheta = - w/ 2\theta$ so that $\hat w=0$.

\begin{table}[t]\centering
  \small
  \begin{tabular}{ccccl}
    \toprule
    \sffamily Cat.
    &\sffamily Parameters
    &$(\dot\vx, \dot t)$
    &\parbox{82pt}{
    $g_\eps(\vx_0, t_0)$\\
    \null\quad$=\bigl(\vx(\eps), t(\eps)\bigr)$
    }
    & \sffamily Description
    \\
    \midrule\rule[-16pt]{0pt}{36pt}
    \sffamily A
    &\parbox{40pt}{\centering $\theta=0$ \\$w=0$ \\ $  \cM\vv=0$}
    &\parbox{68pt}{\centering $\dot\vx = \cM\vx +  \vv$ \\
    $\dot t =0$}
    &\parbox{84pt}{\centering $\vx(\eps) = e^{\eps \cM}\vx_0 +
    \eps \vv$\\
    $t(\eps) = t_0$}
    &\parbox{76pt}{\sffamily Circles and \\ Shrinking Helices}
    \\
    \midrule\rule[-16pt]{0pt}{36pt}
    \sffamily B
    &\parbox{40pt}{\centering $  \theta=0$\\ $  w=1$\\ $ \cM \vv=0$}
    &\parbox{68pt}{\centering $\dot\vx =  \cM\vx +  \vv$\\
    $\dot t =1$}
    &\parbox{84pt}
    {\centering $\vx(\eps) = e^{\eps \cM}\vx_0 +\eps \vv$\\
    $t(\eps) = t_0 + \eps$}
    &\parbox{76pt}{\sffamily Translating and\\ Rotating Solitons}
    \\
    \midrule\rule[-16pt]{0pt}{36pt}
    \sffamily C
    &\parbox{40pt}{\centering $\theta=1$\\      $  w=0$\\ $ \vv=0$}
    &\parbox{68pt}{\centering $\dot\vx = \vx +  \cM\vx $\\
    $\dot t =2t$}
    &\parbox{84pt}{\centering $\vx(\eps) = e^\eps e^{\eps \cM}\vx$\\
    $t(\eps) = e^{2\eps}t_0$}
    &\parbox{76pt}{\sffamily Dilating and\\ Rotating Solitons}
    \\
    \bottomrule
  \end{tabular}
  \medskip
  \caption{A taxonomy of one-parameter subgroups of $\fG$.  Up to
  coordinate changes there are three categories of subgroups of $\fG$.}
  \label{tab:list-of-subgroups}
\end{table}

\subsection*{Dilating space and time}
Finally, we can apply a parabolic dilation, $\vx = \vartheta \hat\vx$,
$t = \vartheta^2\hat t$, with $\vartheta>0$.  If $ \theta\neq0$ then
we have already arranged that $ \vv=0$, $ w=0$ by appropriately
choosing our origin in space-time.  The equations (\ref{eq:x-t-flow})
are in this case linear homogeneous and thus the dilation has no
effect.

If $ \theta=0$ then the dilation turns the $t$-equation into $d\hat
t/d\eps = \vartheta^{-2} w = \hat w$.  Since $w\geq 0$ by assumption
(\ref{eq:can-choose-w-nonneg}) we are left with two options: either $
w>0$ and we can arrange $\hat w=1$ by choosing $\vartheta= w^{-1/2}$,
or $ w=0$.

\bigskip

We conclude from the above discussion that after applying a suitable
coordinate transformation we may always assume that the parameters $
\theta, \vv, w, \cM$ belong to one of three categories, namely,
\textbf{A}: $ \theta= w=0$, $ \cM \vv=0$; \textbf{B}: $ \theta=0$, $
w=1$, $ \cM \vv=0$; and \textbf{C}: $ \theta\neq 0$, $ w=0$, $ \vv=0$.
(See Table~\ref{tab:list-of-subgroups}.)

\section{Circles and Shrinking Helices}
\label{sec:shrinking-helices}
Let $\{g_\eps : \eps\in\R\}$ be a one-parameter subgroup corresponding to
category A and let $\{\vc(t):t\in\R\}$ be a family of curves invariant
under $g_\eps$.  If $(\vC_0, t_0)$ is a point on one of these curves, then
the $g_\eps$ orbit through $(\vC_0, t_0)$ lies in the curve.  Since the
$g_\eps$ do not change the time variable $t$, the $g_\eps$ orbit of any
point sweeps out a curve at time $t$ parametrized by $\eps$.  Therefore a
$g_\eps$ invariant family of curves can be parametrized by choosing one
point $\vC_0(t)$ on the curve for each time $t$, and then letting the group
$g_\eps$ act on that point.  The result is
\begin{equation}
  \vc(\eps, t)  = e^{\eps \cM}\vC_0(t) + \eps \vv.
  \label{eq:helix-parametrization}
\end{equation}
If the vector $ \vv\neq0$ then one can always choose $\vC_0(t)\perp
\vv$.

Recall that we had chosen coordinates in which the matrix of $\cM$ is
given by (\ref{eq:rotation-normal-form}).  The exponential
$e^{\eps\cM}$ therefore acts on $\R^n$ by rotating the $(x_{2j-1},
x_{2j})$ component by an angle $\omega_j\eps$ for $j=1, \dots, m$.  In
particular, in two and three dimensions $e^{\eps\cM}$ performs a
rotation by an angle $\omega_1\eps$ around the origin in $\R^2$ or the
$x_3$-axis in $\R^3$, and we see that $\vc(\eps, t)$ traces out the
circle in the $x_1x_2$-plane through $\vC_0$ if $\vv=0$, or a helix
around the $x_3$-axis if $\vv\neq0$.

In higher dimensions the curve traced out by $\vc(\eps, t)$ depends on
the rotation frequencies $\omega_1, \dots, \omega_m$.  If they are
integer multiples of a common frequency, i.e.~if $\omega_j =
p_j\Omega$ ($\Omega>0$, $p_j\in\N$, $1\leq j\leq m$), then the
exponential $e^{\eps\cM}$ is $2\pi/\Omega$ periodic.  If $\vv=0$ then
the curve $\vc(\eps) = e^{\eps\cM}\vC_0$ is a closed curve on the
sphere with radius $\|\vC_0\|$ in $\R^{2m}\times\{0\} \subset \R^n$.
For instance, for $n=4$, $\omega_1=p_1$, $\omega_2=p_2$, with $p_1,
p_2\in\N$, and for $\vC_0 = (r_1, 0, r_2, 0)$ with $r_1^2+r_2^2=1$,
the curve
\[
\vc(\eps) = \bigl(r_1 \cos p_1\eps, r_1\sin p_1\eps, r_2 \cos p_2\eps,
r_2\sin p_2\eps\bigr)
\]
is a $(p_1, p_2)$ torus knot in $\Sph^3$.

If the rotation rates $\omega_j$ are not multiples of some common frequency,
then the curve $\vc(\eps) = e^{\eps\cM}\vC_0$ is not closed, but instead is
dense in a torus whose dimension is the largest number of $\omega_j$ that are
independent over the rationals.

If $\vv\neq0$, then the projection of $\vc(\eps) = e^{\eps\cM}\vC_0+\eps\vv$
onto $\R^{2m}\times\{0\}$ is the curve $\tilde \vc(\eps) = e^{\eps\cM}\vC_0$.
Therefore $\vc(\eps) = e^{\eps\cM}\vC_0+\eps\vv$ describes a generalized helix
over $\tilde \vc$.

Suppose now that $\vc(\eps, t)$ given by (\ref{eq:helix-parametrization}) is a
family of curves evolving by Curve Shortening that is invariant under the one
parameter group $g_\eps$.  Then Curve Shortening implies a differential equation
for the time evolution of $\vC_0(t)$, which we now derive.

Direct computation shows that
\[
\vc_\eps = \cM e^{\eps \cM}\vC_0(t) + \vv,\quad \vc_{\eps\eps} = \cM^2
e^{\eps \cM}\vC_{0}(t),\quad \vc_t = e^{\eps \cM}\vC_0'(t),
\]
where ${}'$ is the derivative with respect to time.

Since $\cM$ is antisymmetric and since $ \cM \vv=0$ we have $e^{\eps \cM} \vv =
\vv$.  From there we see that the two terms in $\vc_\eps$ are orthogonal.
Therefore
\[
\|\vc_\eps\|^2 = \|\vv\|^2 + \|\cM\vC_0(t)\|^2.
\]
Hence the arc length derivative and derivative with respect to $\eps$ are
related by
\[
\frac {d}{ds} = \frac1{\|\vc_\eps\|} \frac d{d\eps} = \frac {1}{\sqrt{
\|\vv\|^2 + \|\cM\vC_0(t)\|^2}} \frac d{d\eps} .
\]
This allows one to compute
\[
\vc_{ss} = \frac{\vc_{\eps\eps}} {\|\vc_\eps\|^2} =
e^{\eps\cM}\frac{\cM^2 \vC_0(t)}{\|\vv\|^2 + \|\cM\vC_0(t)\|^2}.
\]
Together with $\vc_t = e^{\eps\cM}\vC_0'(t)$ this tells us that the family of
curves (\ref{eq:helix-parametrization}) evolves by Curve Shortening if
\[
\vc_t = \vc_{ss}, \mbox{ i.e.\ if } e^{\eps\cM}\vC_0'(t) =
e^{\eps\cM}\frac{\cM^2 \vC_0(t)}{\|\vv\|^2 + \|\cM\vC_0(t)\|^2}.
\]
Cancelling $e^{\eps\cM}$ on both sides we see that (\ref{eq:CS}) implies
\begin{equation}
  \vC_0'(t) = \frac{\cM^2 \vC_0(t)}{\|\vv\|^2 + \|\cM\vC_0(t)\|^2}.
  \label{eq:cs-for-helices}
\end{equation}
This equation is essentially linear, after a time-change.  If one introduces a
new time variable $\tau$, related to $t$ by
\begin{equation}
  \frac {dt}{d\tau} = \|\vv\|^2 + \|\cM\vC_0(t)\|^2,
  \label{eq:tau-t-relation}
\end{equation}
then $\vC_0$, as a function of $\tau$ satisfies
\[
\frac{d\vC_0} {d\tau} = \cM^2 \vC_0.
\]
Since $\cM$ is given by (\ref{eq:rotation-normal-form}), we have
\[
\cM^2 =
\left(
\begin{array}{cccc}
  -\omega_1^2 \identity_{\R^2} &&&\\
  &\ddots&&\\
  &&-\omega_m^2\identity_{\R^2}&\\
  &&&0\\
\end{array}
\right), \qquad \identity_{\R^2} =
\left(
\begin{array}{cc}
  1 & 0 \\ 0 & 1
\end{array}
\right)
\]
It follows that
\begin{equation}
  \vC_0(\tau) = \sum_{j=1}^m  e^{-\omega_j^2\tau} \vC_j
  \label{eq:helix-radius-evolution}
\end{equation}
for certain constant vectors $\vC_j\in E_j$ ($E_j$ as in
(\ref{eq:Rn-cA-invariant-splitting})).  To eliminate $\tau$, we observe that
\[
\|\vv\|^2 + \|\cM\vC_0\|^2 = \|\vv\|^2 + \sum_{j=1}^m \omega_j^2
e^{-2\omega_j^2\tau}\|\vC_j\|^2
\]
and integrate (\ref{eq:tau-t-relation}) to get
\begin{equation}
  t = \|\vv\|^2 \tau
  - \sum_{j=1}^m \frac{\omega_j}{2}e^{-2\omega_j^2\tau} \|\vC_j\|^2.
  \label{eq:helices-t-tau-relation}
\end{equation}
One cannot explicitly solve this equation for $t$, but one can find approximate
solutions when $\tau\to\pm\infty$.  Assume that $\omega_1\leq\omega_2\leq \cdots
\leq \omega_m$, and let $k$ and $l$ be the smallest and larges values of $j$
respectively for which $\vC_j\neq0$.

As $\tau\to -\infty$ we can approximate $t$ by the largest term in
(\ref{eq:helices-t-tau-relation}), giving us
\[
t= -\bigl(\frac{\omega_l}{2} + o(1)\bigr)
e^{-2\omega_l^2\tau}\|\vC_l\|^2,
\]
while, by (\ref{eq:helix-radius-evolution}),
\[
\vC_0 = e^{-\omega_l^2 \tau} \vC_l + o\bigl(e^{-\omega_l^2
\tau}\bigr).
\]
These expansions for $\vC_0$ and $t$ show that as $\tau\to -\infty$, the actual
time $t$ also goes to $-\infty$, and
\begin{equation}
  \vC_0(t) = \sqrt{-\frac{2 t}{\omega_l}}\; \frac{\vC_l} {\|\vC_l\|}
  + o\bigl(\sqrt{-t}\bigr).
  \label{eq:helix-at-t-neginfty}
\end{equation}
Thus for $t\to-\infty$ the rescaled curves $(-t)^{-1/2}\vc(\eps, t)$ converge to
a circle with radius $\sqrt{2/\omega_l}$ in the $E_{l}$ plane.

Examples: If $\vC_l$ is the only non zero $\vC_j$, then $(-t)^{-1/2}\vc$
actually coincides with this circle, but if there are other nonzero terms in
(\ref{eq:helix-radius-evolution}), then when the $\omega_j=p_j\Omega$ are
integer multiples of a common rate $\Omega$, the normalized curve
$(-t)^{-1/2}\vc(\cdot, t)$ is a closed curve converging to a multiple cover of
the circle with radius $\sqrt{2\omega_l}$; if at least two $\omega_j$ with
$\vC_j\neq0$ are independent over $\Q$, then $(-t)^{-1/2} \vc(\cdot, t)$ is
densely wound on a torus, which converges to the circle in $E_l$ with radius
$\sqrt{2\omega_l}$ as $t\to-\infty$.

The behaviour for $\tau\to\infty$ depends on whether or not $\vv=0$.

If $\vv\neq0$, then (\ref{eq:helices-t-tau-relation}) implies that
\[
t = \bigl(\|\vv\|^2+o(1)\bigr) \tau, \qquad (\tau\to\infty).
\]
By (\ref{eq:helix-radius-evolution}), $\vC_0(t)$ decays exponentially as
$t\to\infty$, so that (\ref{eq:helix-parametrization}) tells us that the curve
$\vc(\eps, t)$ converges at an exponential rate to the straight line $\vc(\eps,
\infty) = \eps\vv$.

On the other hand, if $\vv=0$, then (\ref{eq:helices-t-tau-relation}) implies
that
\[
t = -\bigl(\frac{\omega_k}2 + o(1) \bigr) e^{-\omega_k^2\tau}.
\]
Reasoning as above we find that
\begin{equation}
  \vC_0(t) = \sqrt{-\frac{2 t}{\omega_k}}\; \frac{\vC_k} {\|\vC_k\|} +
  o\bigl(\sqrt{-t}\bigr).
  \label{eq:vC0-at-t-zero}
\end{equation}
Thus as $t\nearrow 0$ the rescaled curves $(-t)^{-1/2}\vc(\eps, t)$ converge to
a circle with radius $\sqrt{2/\omega_k}$ in the $E_k$ plane.

\subsection{Connection with Curve Shortening on $\Sph^{n-1}$}
\label{sec:CS-Sn-1}

The solutions we have found for $\vv=0$ are related equivalent to certain
special solutions of Curve Shortening on $\Sph^{n-1}$.

\begin{lemma}
  If $\vC_0(t)$ is a solution of (\ref{eq:cs-for-helices}) and if $\vv=0$, then
  $\|\vC_0(t)\|^2 = 2(T-t)$ for some $T\in \R$.  The corresponding curve $\vc(t,
  \eps) = e^{\eps\cM}\vC_0(t)$ therefore lies on a sphere of radius $\sqrt{2(T-t)}$.
\end{lemma}
\begin{proof}
We compute, using (\ref{eq:cs-for-helices}) and  $\cM^t = -\cM$,
  \[
  \frac{d} {dt}\|\vC_0(t)\|^2
  = 2 \frac{\langle\vC_0(t), \cM^2\vC_0(t)\rangle} {\|\cM\vC_0(t)\|^2}
  = -2 \frac{\langle\cM\vC_0(t), \cM\vC_0(t)\rangle} {\|\cM\vC_0(t)\|^2}
  =-2,
  \]
  which immediately implies the Lemma.
\end{proof}

\begin{lemma}
  Let $\vc:(\infty, T)\times\R \to \R^n$ be a family of curves with $\|\vc(t,
  \eps) \|^2 = 2(T-t)$, and consider the family of curves $\tvc$ on the unit
  sphere given by
  \begin{equation}
    \vc(t, \eps) = \sqrt{2(T-t)} \; \tvc\bigl(-\tfrac12\ln(T-t), \eps\bigr).
    \label{eq:CS-on-sphere}
  \end{equation}
  Then $\vc$ evolves by Curve Shortening if and only if $\tvc$ evolves Curve
  Shortening on $\Sph^{n-1}$.
\end{lemma}
\begin{proof}
  Since Curve Shortening for immersed curves in $\Sph^{n-1}$ is given by
  $\tvc_\theta = \tvc_{ss} + \tvc$, the lemma follows from a direct computation:
  \[
  \vc_t = \frac{1} {\sqrt{2(T-t)}}\bigl\{\tvc_{\theta} - \tvc\bigr\},\qquad
  \vc_{ss} = \frac{\tvc_{ss}} {\sqrt{2(T-t)}},
  \]
  where $\theta = -\frac12\ln(T-t)$.
\end{proof}

Putting these two Lemmas together we see that if $\vc(t, \eps) =
e^{\eps\cM}\vC_0(t)$ is a solution to \eqref{eq:CS}, then for suitably chosen
$T$ the family of curves $\tvc(\theta, \eps)$ defined by \eqref{eq:CS-on-sphere}
is an eternal solution of Curve Shortening on $\Sph^{n-1}$.  The asymptotics
derived in \eqref{eq:helix-at-t-neginfty} and \eqref{eq:vC0-at-t-zero} imply
that $\tvc$ converges to a cover of some great circle on $\Sph^{n-1}$, both for
$\theta\to \infty$ and for $\theta\to -\infty$.

\section{Solitons}
\label{sec:solitons}
The one-parameter symmetry groups that led us to the shrinking helices left the
time coordinate $t$ invariant.  From here on we will consider one-parameter
symmetry groups $\{g_\eps\}$ of $\fG$ of categories B and C (see
Table~\ref{tab:list-of-subgroups}), for which $dt/d\eps \neq0$.

In this section we will derive an ODE whose solutions generate all solutions to
(\ref{eq:CS}) that are invariant under a one-parameter subgroup of $\fG$ of
category B or C.  In both cases we start with a family of curves $\vc(\xi, t)$.
Assuming the family is invariant under a category B or C subgroup $\{g_\eps\}$
we find an equation relating $\vc_t$ and $\vc_\xi$.  Combining this with the
further assumption that $\vc$ is a solution of (\ref{eq:CS}), i.e.~of
$\pro_{\vc_s}\vc_t = \vc_{ss}$, we can eliminate the time derivative $\vc_t$
after which we are left with an ODE.  This ODE must be satisfied by the curve
$\vc(\cdot, t)$ at all times $t$.

\subsection{The ODE for Translating and Rotating Solitons}
If $\{g_\eps\}$ is of category B, then we have $\theta=0$, $w=1$, $\cM\vv=0$,
and the subgroup is given by 
\[
g_\eps(\vx, t) = \bigl(e^{\eps\cM}\vx + \eps\vv, t+\eps\bigr).
\]
In words: $g_\eps$ advances time by $\eps$, and translates and rotates space by
$\eps\vv$ and $e^{\eps\cM}$, respectively.

If a family of curves $\vc(\xi, t)$ is known at time $t=t_0$, and if the family
is known to be invariant under the action of the subgroup $g_\eps$, then this
action determines the curves at other times.  Thus for any given $\vc(\xi_0,
t_0)$ the point $g_\eps \bigl(\vc(\xi_0, t_0), t_0\bigr)$ also belongs to the
family of curves, so that a $\xi_1$ exists for which
\[
g_\eps\bigl(\vc(\xi_0, t_0), t_0\bigr) = \bigl( \vc(\xi_1, t_0+\eps),
t_0+\eps \bigr)
\]
holds.  Different choices of $\eps$ will lead to different values of $\xi_1$.
We denote the function $\eps\mapsto \xi_1$ defined in this way by $\xi_1 =
\xi(\eps)$.  Consequently we have for all $\eps$
\begin{equation}
  \vc(\xi(\eps), t_0+\eps)
  = e^{\eps\cM}\vc(\xi_0, t_0) + \eps\vv.
  \label{eq:type-II-soliton-evolution}
\end{equation}
Since $\xi(\eps)$ is obtained by solving this equation, and since $\xi \mapsto
\vc(\xi, t)$ is an immersion, the Implicit Function Theorem implies that
$\xi(\eps)$ depends smoothly on $\eps$.  We can therefore differentiate
(\ref{eq:type-II-soliton-evolution}) with respect to $\eps$ at $\eps=0$, which
leads to
\[
\vc_t + \frac{d\xi} {d\eps}\, \vc_\xi = \cM \vc + \vv,
\]
hence
\[
\pro_{\vc_s}\vc_t = \pro_{\vc_s}\bigl(\cM \vc + \vv\bigr)
\]
and, because of (\ref{eq:CS}),
\[
\vc_{ss} = \pro_{\vc_s}\bigl(\cM \vc + \vv\bigr).
\]
This equation contains no time derivatives and is in fact an ODE.  Therefore we
find that at any time $t_0$ the parametrized curve $\vC(\xi) = \vc(\xi, t_0)$ is
a solution of the ODE
\begin{equation}
  \vC_{ss} = \pro_{\vc_s}\bigl(\cM \vC + \vv\bigr)
  \label{eq:type-II-soliton}
\end{equation}
Conversely, if a parametrized curve $\vC:\R\to\R^n$ satisfies
(\ref{eq:type-II-soliton}), then the family of curves defined by
(\ref{eq:type-II-soliton-evolution}) satisfies (\ref{eq:CS}).

\subsection{The ODE for Dilating and Rotating Solitons}
If $\{g_\eps\}$ is of category C then we have $\theta=1$, $w=0$, $\vv=\vvv0$.
The action of $g_\eps$ is given by
\[
g_\eps(\vx, t) = \bigl(e^\eps e^{\eps\cM}\vx, e^{2\eps}t\bigr).
\]
If $\vc(\xi, t)$ is a solution of (\ref{eq:CS}) that is invariant under
$\{g_\eps\}$, and if it is defined at some time $t_0\in\R$, then, reasoning as
in the category B case above, we find
\[
g_\eps\bigl( \vc(\xi_0, t_0), t_0 \bigr) = \Bigl( \vc(\xi(\eps),
e^{2\eps}t_0) , e^{2\eps}t_0 \Bigr)
\]
for some $\xi(\eps)$ that depends smoothly on $\eps$.  Consequently,
\begin{equation}
  \vc(\xi(\eps), e^{2\eps}t_0)
  =
  e^\eps e^{\eps\cM}\vc(\xi_0, t_0).
  \label{eq:type-III-soliton-evolution}
\end{equation}
Once again we can differentiate with respect to $\eps$ and set
$\eps=0$, this time with result
\[
2t_0 \vc_t + \frac{d\xi}{d\eps}\, \vc_\xi = \vc + \cM\vc.
\]
Taking the normal component, and using (\ref{eq:CS}) we arrive at the ODE
\begin{equation}
  2t_0 \vc_{ss}=
  \pro_{\vc_s} \bigl(\vc + \cM\vc\bigr).
  \label{eq:typeIII-epsderiv}
\end{equation}
We now have to distinguish among the possibilities $t_0=0$ and $t_0\neq0$.

If $t_0=0$, then (\ref{eq:typeIII-epsderiv}) is not a second order differential
equation at all.  However, (\ref{eq:type-III-soliton-evolution}) implies that
for some smooth $\eps\mapsto \xi(\eps)$ one has
\[
\vc(\xi(\eps) , 0) = e^\eps e^{\eps\cM} \vc(\xi_0, 0).
\]
The curve at time $t_0=0$ is a reparametrization of a generalized logarithmic
spiral: it is traced out by $\eps\mapsto e^\eps e^{\eps\cM} \vc_0$ for some
constant vector $\vc_0$.

If $t_0\neq 0$, then we can divide both sides of (\ref{eq:typeIII-epsderiv}) by
$2t_0$ and use (\ref{eq:CS}) to obtain an ODE for the curve at time $t_0$.
Introducing the constants
\begin{equation}
  \alpha = \frac{1}{2t_0}, \mbox{ and }
  \cA = \frac{1}{2t_0}\cM
  \label{eq:alpha-A-def}
\end{equation}
we find that the parametrized curve $\vC(\xi) = \vc(\xi, t_0)$ satisfies
\begin{equation}
  \vC_{ss} = \pro_{\vc_s} \bigl(\alpha\vC + \cA\vC\bigr)
  \label{eq:type-III-soliton}
\end{equation}
whenever $\vc$ is a $\{g_\eps\}$-invariant solution of (\ref{eq:CS}).

Conversely, if $\vC:\R\to\R^n$ is a given solution to
(\ref{eq:type-III-soliton}) for some $\alpha\neq0$ and some
$\cA\in\mathfrak{so}(\R^n)$, then upon defining $t_0 = 1/2\alpha$,
$\cM=2t_0\cA$, one can use (\ref{eq:type-III-soliton-evolution}) to extend the
curve $\vC$ to a $\{g_\eps\}$-invariant family of curves $\vc(\xi, t)$ that
satisfies (\ref{eq:CS}).  We have not specified how to choose $\xi(\eps)$ in
(\ref{eq:type-III-soliton-evolution}), but in view of the invariance of Curve
Shortening under reparametrization of the curves, the choice of $\xi(\eps)$ is
immaterial.  One could simply set $\xi(\eps) = \xi_0$ for all $\eps$, so that
our extension through (\ref{eq:type-III-soliton-evolution}) of $\vC$ would be
given by
\[
\vc(\xi, t) = e^\eps e^{\eps \cM} \vC(\xi), \mbox{ with } \eps =
\frac12 \log \frac{t} {t_0}.
\]
This family of curves is defined for all $t$ that have the same sign as $t_0$.
Thus if a given curve $\vC$ satisfies (\ref{eq:type-III-soliton}) with
$\alpha<0$, then the resulting family of curves is defined for all $t<0$ and
represents a rotating and \textit{shrinking} soliton.  If $\alpha>0$ then $\vc$
is defined for all $t>0$ and represents a rotating and \textit{expanding}
soliton.

\subsection{The general soliton equation}
\label{sec:general-soliton-equation}%
Equations (\ref{eq:type-II-soliton}) and (\ref{eq:type-III-soliton}) are both
special cases of (\ref{eq:soliton-profile}), i.e.
$  \vC_{ss} = (\alpha + \cA) \vC + \vv  + \lambda \vC_s$.
Here $\lambda$ is a smooth function of the parameter and is uniquely determined
by the requirement that $\vC_s \perp \vC_{ss}$.  Comparing the inner product
with $\vC_s$ of both sides of (\ref{eq:soliton-profile}) leads to
\begin{equation}
  \lambda = -\bigl\langle (\alpha + \cA) \vC + \vv, \vC_s \bigr\rangle.
  \label{eq:lambda-revealed}
\end{equation}
We can also write (\ref{eq:soliton-profile}) in the following form
\begin{equation}
  \vC_{ss} = \pro_{\vT}\bigl((\alpha + \cA) \vC + \vv\bigr),
  \label{eq:soliton-profile-w-pro}
\end{equation}
where $\vT = \vC_s$ is the unit tangent to the curve $\vC$.

The soliton equation (\ref{eq:soliton-profile}) defines a first order system of
differential equations for $(\vC, \vT)$,
\begin{equation}
  \vC_s = \vT,
  \qquad
  \vT_s = \pro_\vT \bigl[(\alpha+\cA)\vC + \vv\bigr].
  \label{eq:soliton-profile-as-system}
\end{equation}
Since $\vC\in\R^n$ can be arbitrary and $\vT$ is always a unit vector, this
system defines a flow on $\R^n\times \Sph^{n-1}$.  It is this system that we
used to produce the soliton images in this paper.

\section{Components in the Null space and range of $\cA$}
\label{sec:Components-in-nulspace}
In this section we split $\R^n$ into the null space and range of $\cA$, i.e.
$\R^n= N(\cA) \oplus R(\cA)$, and we consider the corresponding components of a
soliton.  Thus, for a given soliton we write
\begin{equation}
  \vC(s) = \vZ(s) + \vW(s), \qquad
  \vZ(s)\in N(\cA), \; \vW(s) \in R(\cA).
  \label{eq:NA-RA-decomposition}
\end{equation}
It is well known that non-rotating dilating or non-rotating translating solitons
(i.e.~those with $\cA=0$) are planar.  This fact continues to be true in a
weaker sense in the case where the rotation matrix $\cA$ does not necessarily
vanish. Here we show that the projection $\vZ$ of any rotating-dilating or
rotating-translating soliton onto the null space $ N(\cA)$ is planar.  In the
case of purely rotating solitons, the projection $\vZ$ is in fact at most a line
segment, and the soliton is a graph over this line segment.

\begin{figure}[t]
  \centering
  \includegraphics[width=\textwidth]{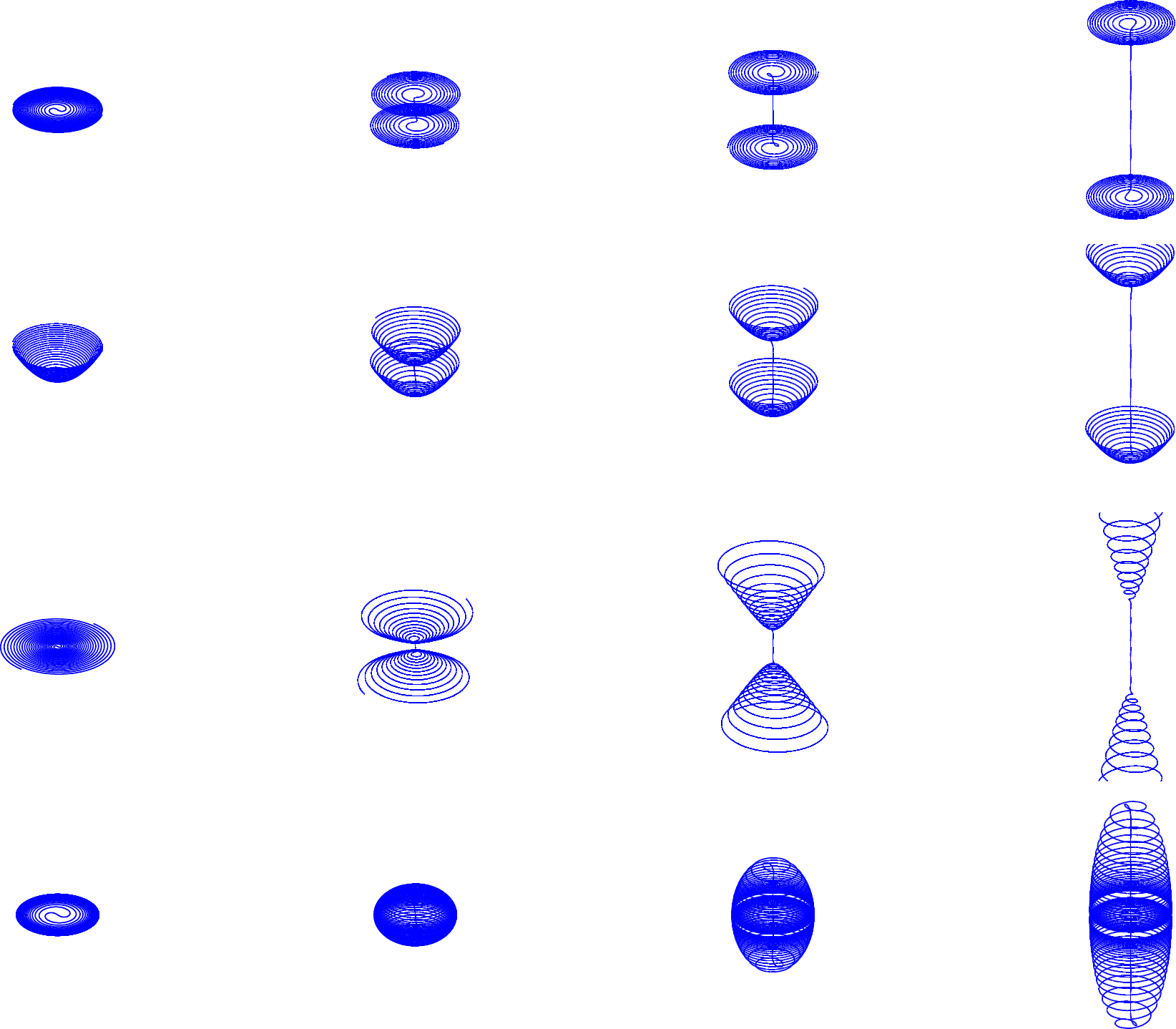}
  \caption{\textbf{Examples of non-planar solitons involving rotation.}  Shown
  are sixteen different solitons that rotate around the $z$-axis (top); rotate
  and translate upward along the $z$-axis (second row); rotate around the
  z-axis and expand (3rd row); rotate and shrink (bottom). }
  \label{fig:solitons-and-nullspace}
\end{figure}

\begin{lemma}
  \label{lem:dilating-rotating-is-2D-in-NA} Let $\vC$ be dilating-rotating
  soliton satisfying (\ref{eq:soliton-profile}) with $\vv=0$.  Then the
  orthogonal projection of $\vC$ onto the null space $N(\cA)$ of $\cA$ is
  contained in a two dimensional subspace of $N(\cA)$.

  In particular, all purely dilating solitons ($\alpha\neq0$, $\cA=0$, $\vv=0$)
  are planar: they are copies of the Abresch-Langer curves.
\end{lemma}

\begin{proof}
  $\vC$ satisfies (\ref{eq:soliton-profile}). Therefore, if $\Pi_N$ denotes the
  orthogonal projection onto $N(\cA)$, then $\Pi_N\cA = \cA\Pi_N = 0$ and thus
  $\vZ=\Pi_N\vC$ satisfies
  \[
  \vZ_{ss} = \alpha \vZ + \lambda(s)\vZ_s.
  \]
  Let $\Phi_1(s)$ and $\Phi_2(s)$ be the two solutions of the general solution
  of the second order scalar linear differential equation $\Phi''= \alpha\Phi +
  \lambda(s) \Phi'$ with $\Phi_1(0) = 1$, $\Phi_1'(0) = 0$ and $\Phi_2(0) = 0$,
  $\Phi_2'(0) = 1$.  Then $\vZ$ is given by
  \[
  \vZ(s) = \Phi_1(s)\vZ(0) + \Phi_2(s) \vZ'(0).
  \]
  It follows that $\vZ(s)$ is contained in the two dimensional subspace spanned
  by $\vZ(0)$ and $\vZ'(0)$.
\end{proof}

\begin{lemma}
  \label{lem:translating-rotating-is-2D-in-NA} Let $\vC$ be translating-rotating
  soliton satisfying (\ref{eq:soliton-profile}) with $\alpha=0$.  Then the
  orthogonal projection of $\vC$ onto the null space $N(\cA)$ of $\cA$ is
  contained in a two dimensional affine subspace of $N(\cA)$.

  In particular, all purely translating solitons ($\alpha = 0$, $\cA=0$,
  $\vv\neq0$) are planar: they are copies of the ``Grim Reaper.''
\end{lemma}
\begin{proof}
  We again consider the projection $\vZ=\Pi_N\vC$.  Since we always assume
  $\vv\in N(\cA)$, we have $\Pi_N\vv = \vv$.  Therefore $\vZ$ satisfies
  \[
  \vZ_{ss} = \vv + \lambda(s)\vZ_s
  \]
  where $s$ is still the arclength along the curve $\vC$.  The general solution
  of the (scalar) ODE
  \[
  y'' - \lambda(s) y' = 1
  \]
  can be written as
  \[
  y(s) = y(0) + y'(0)\Phi_1(s) + \Phi_p(s),
  \]
  where $\Phi_1$ is the solution of the homogeneous equation with $\Phi_1(0) =
  0$ and $\Phi_1'(0) = 1$, while $\Phi_p$ is the particular solution with
  $\Phi_p(0) = \Phi_p'(0) = 0$.

  It follows that
  \begin{equation}
    \vZ(s) = \vZ(0) + \Phi_1(s) \vZ'(0) + \Phi_p(s) \vv.
    \label{eq:Z-for-trans-rotating-solitons}
  \end{equation}
  Therefore $\vZ(s)$ is always contained in the plane through $\vZ(0)$ spanned
  by the vectors $\vZ'(0)$ and $\vv$.

\end{proof}
\begin{corollary}
  Let $\vT(s)$ be a smooth family of unit vectors, and let 
  \[
  \vC(s) = \vC(0) + \int_0^s \vT(s')\; ds'
  \]
  be the corresponding curve in $\R^n$.
  If $\vT(s)$ satisfies the differential equation
  \begin{equation}
    \vT_s = \pro_\vT(\vv)
    \label{eq:grim-reaper-tangent}
  \end{equation}
  for some fixed vector $\vv$, then  the curve $\vC$ is a purely
  translating soliton with velocity vector $\vv$.
\end{corollary}
\begin{lemma}
  \label{lem:purely-rotatin}
  Let $\vC$ be a purely rotating soliton.  Then its projection $\vZ$
  onto the null space $N(\cA)$ is contained in some line $\ell\subset
  N(\cA)$.

  The projection consists of either one point, a half line obtained by
  splitting $\ell$ at a point $P\in \ell$, or the line segment between
  two points $P, Q\in\ell$.

  If the projection is a half line or a line segment $m\subset\ell$,
  then the soliton is a graph over $m$ in $m\times R(\cA)$.

  The projected soliton approaches any endpoint of $m$ at an
  exponential rate, i.e.~if $\lim_{s\to\pm\infty}\vZ(s) = P$ then
  $\|\vZ(s) - P\|\leq Ce^{-\delta|s|}$ for some $C, \delta>0$.
\end{lemma}
This is illustrated in Figure~\ref{fig:solitons-and-nullspace} (top row).
Non-planar rotating solitons in $\R^3$ are graphs over the rotation axis, and if
they are asymptotic to some plane orthogonal to the rotation axis, then they
approach this plane at an exponential rate.
\begin{proof}
  By the previous Lemma the soliton is given by
  (\ref{eq:Z-for-trans-rotating-solitons}) where in our case $\vv=0$.  Thus the
  projection onto $N(\cA)$ of a purely rotating soliton is given by
  \[
  \vZ(s) = \vZ(0) + \Phi_1(s) \vZ'(0),
  \]
  where $\Phi_1$ satisfies $\Phi_1''=\lambda(s)\Phi_1'$ with $\Phi_1(0) = 0$,
  $\Phi_1'(0) = 1$.  The projected soliton is therefore contained in the line
  $\ell$ through $\vZ(0)$ with direction $\vZ'(0)$.  If $\vZ'(0)=0$ then the
  projection reduces to just the one point $\vZ(0)$.

  Assume from here on that $\vZ'(0)\neq0$.  The line $\ell$ is parametrized by
  \[
  \vZ= \vZ(0) + \phi\vZ'(0), \qquad \phi\in\R.
  \]
  The equation for $\Phi_1$ implies that $\Phi_1'(s)>0$ for all $s$.  One way to
  see this is to observe that $\Phi_1'$ satisfies a linear homogeneous first
  order differential equation so that either $\Phi_1'(s) = 0$ for all $s$, or
  $\Phi_1'(s)\neq0$ for all $s$.  Alternatively, one can just solve the ODE for
  $\Phi_1$, resulting in
  \begin{equation}
    \Phi_1(s) = \int_0^s e^{\int_0^{s'}\lambda(s'')ds''}\; ds'.
    \label{eq:Phi1}
  \end{equation}
  Since $\Phi_1$ is a monotone function the map $s\mapsto \vZ(s)$ is one-to-one.
  The range of $\Phi_1$ is an interval $(\phi_-, \phi_+)$ where
  \[
  \phi_\pm = \lim_{s\to\pm\infty} \Phi_1(s)
  \]
  and the projection of the soliton is the line segment $m\subset\ell$
  consisting of all points $\vZ(0)+\phi\vZ'(0)$ with $\phi_-<\phi<\phi_+$.  For
  each $\phi\in(\phi_-, \phi_+)$ one can find $s=\Phi_1^{-1}(\phi)$ and write
  \[
  \vC(s) = \vZ(0) + \phi\vZ'(0) + \vW,
  \]
  where $\vW$ is the projection of $\vC(s)$ onto $N(\cA)^\perp = R(\cA)$.
  Regarding $\vW$ as a function of $\phi\in(\phi_-, \phi_+)$ we see that the
  soliton is indeed a graph over the line segment $m$.

  We now estimate the rate at which $\Phi_1(s)\to\phi_\pm$ as $s\to\pm\infty$
  from (\ref{eq:Phi1}), which requires us to look at $\lambda(s)$ given by
  (\ref{eq:lambda-revealed}).  Under the assumptions of this lemma
  (\ref{eq:lambda-revealed}) reduces to $ \lambda(s) = -\langle\vC_s,
  \cA\vC\rangle$.  Differentiate, use (\ref{eq:soliton-profile}), and keep in
  mind that $\vC_s\perp\vC_{ss}$\ :
  \[
  \lambda'(s) = -\langle\vC_{ss}, \cA\vC\rangle =-\langle\vC_{ss}, \vC_{ss} -
  \lambda\vC_s\rangle =-\|\vC_{ss}\|^2.
  \]
  Thus $\lambda$ is strictly decreasing, except at points of zero curvature on
  the soliton.

  We distinguish between the following possibilities for the function
  $\lambda(s)$.

  \textbf{$\lambda(s) = 0$ for all $s$. }  In this case $\lambda'(s)=0$ and
  consequently $\vC_{ss}=0$ for all $s$.  The soliton must be a straight line.

  \textbf{$\lambda(s)$ changes sign, i.e.~$\lambda(s_{*}) > 0 > \lambda(s_{**})$
  for certain $s_{*} < s_{**}$. }  Let $\lambda(s_{**}) = - \delta$.  Then
  $\lambda(s)\leq -\delta < 0$ for all $s\geq s_{**}$.  Hence
  $\Phi_1'(s)/\Phi_1'(s_{**}) = e^{\int_{s_{**}}^s\lambda(s')ds'} \leq
  e^{-\delta (s-s_{**})}$.  Integrating we find that $\phi_+$ is finite, and
  that $|\Phi_1(s) - \phi_+|\leq C e^{-\delta s}$ for all $s\geq s_{**}$.  An
  analogous argument applies to the behaviour of $\Phi_1(s)$ for $s\to-\infty$.

  \textbf{$\lambda(s)$ does not change sign, and $\lambda(s_*)\neq0$ for some
  $s_*\in\R$. }  Without loss of generality we assume that $\lambda(s)\geq 0$
  for all $s$.  As in the previous case where $\lambda(s)$ changed sign we find
  that $\Phi_1(s)$ converges exponentially as $s\to-\infty$.  We must still
  examine the behaviour of the soliton as $s\to+\infty$.

  The components $\vZ$ and $\vW$ of $\vC$ satisfy linear differential equations
  \begin{eqnarray}
    \vW_{ss} - \lambda(s) \vW_s &=& \cA\vW,
    \label{eq:W-linear-ode} \\
    \vZ_{ss} - \lambda(s) \vZ_s &=& 0.
  \end{eqnarray}
  Multiply the second equation with the integrating factor and integrate.  We
  get
  \begin{equation}
    \vZ_s(s) = e^{\int_{s_0}^s \lambda(s')ds'} \vZ_s(s_0).
    \label{eq:Zs-integrated}
  \end{equation}
  We can choose $s_0$ so that $\vZ_s(s_0)\neq0$, for if we could not, then
  $\vZ_s(s)=0$ for all $s$, implying that $\vZ(s)$ is constant, in which case
  the soliton would lie in an affine subspace parallel to $R(\cA)$.

  Since $\vZ_s$ is a projection of a unit vector, (\ref{eq:Zs-integrated})
  implies that
  \[
  \int_{s_0}^s \lambda(s')ds' \leq -\ln \|\vZ_s(s_0)\|
  \]
  for all $s>s_0$.  The integrand $\lambda$ is nonnegative, so we can take the
  limit $s\nearrow\infty$ to get
  \begin{equation}
    \int_{s_0}^\infty \lambda(s)ds \leq -\ln \|\vZ_s(s_0)\|.
    \label{eq:lambda-integrable}
  \end{equation}
  Recalling that $\lambda(s)$ is non-increasing we conclude that
  $\lim_{s\to\infty} \lambda(s) = 0$.

  It follows from  (\ref{eq:Zs-integrated}) and (\ref{eq:lambda-integrable}) that
  \[
  \vZ_s(+\infty) = \lim_{s\to\infty} \vZ_s(s)
  \]
  exists, and is non-zero.

  Now that we know that $\lambda(s) = o(1)$ as $s\to\infty$, we see that the
  equation satisfied by $\vW$ is a small perturbation of the constant
  coefficient equation
  \begin{equation}
    \vW_{ss} - \cA\vW=0.
    \label{eq:W-linearized-ode}
  \end{equation}
  The eigenvalues of $\cA$, as a linear transformation on $R(\cA)$ are $\{\pm
  i\omega_1, \ldots , \pm i\omega_m\}$.  It follows that the characteristic
  exponents of the equation (\ref{eq:W-linearized-ode}) are
  \[
  \pm \sqrt{\pm i\omega_k} = \bigl(\pm1\pm i\bigr) {\sqrt\frac{\omega_k}{2}},
  \qquad (k=1, \ldots, m).
  \]
  We note that none of these exponents lie on the imaginary axis so that the
  associated system is hyperbolic.  A perturbation theorem
  \cite{CoddingtonLevinson} therefore implies that the general solution of
  (\ref{eq:W-linear-ode}) can be written as
  \[
  \vW(s) = \vW_+(s) + \vW_-(s)
  \]
  where $\vW_+(s)$ consists of exponentially growing terms, and $\vW_-(s)$ is
  exponentially decaying.

  We claim that the exponentially growing component $\vW_+$ vanishes. Indeed,
  \[
  \lambda'(s) = -\|\vC_{ss}\|^2 = -\|\vW_{ss}\|^2 -\|\vZ_{ss}\|^2
  \]
  implies that $\int^\infty \|\vW_{ss}\|^2ds <\infty$.  Integrating twice, it
  follows that $\|\vW(s)\| \leq C s^{3/2}$.  Since $\vW$ and $\vW_+$ differ by
  the exponentially decaying term $\vW_-$ we see that $\|\vW_{+}\|$ also grows
  at most like $\cO(s^{3/2})$.  As $\vW_+$ grows exponentially this can only
  happen if $\vW_+$ vanishes.
\end{proof}

\section{The Distance to the Origin and to the Null space of $\cA$}
\label{sec:distance-to-Origin}

\subsection{A differential equation for distance functions}
\label{sec:diff-equat-dist}

If $\vC:\R\to\R^n$ is a soliton, i.e.~a solution of~(\ref{eq:soliton-profile}),
then we again split $\vC = \vW+\vZ$, as in~(\ref{eq:NA-RA-decomposition}).  The
lengths of these components satisfy various monotonicity properties, depending
on the parameters $\alpha,\cA, \vv$.

\begin{lemma}
  \label{lem:length-of-W-ode}
  If $\cB:\R^n\to\R^n$ is linear with $\cA\cB = \cB\cA$, $\cB\vv = 0$, then
  the quantity
  \[
  \delta(s) = \tfrac12\|\cB\vC\|^2
  \]
  satisfies
  \begin{equation}
    \delta_{ss} - \lambda(s)\delta_s - 2\alpha\delta = 
    \|\cB\vC_s\|^2.
    \label{eq:length-of-BC-ode}
  \end{equation}
  In particular, $ \delta(s) = \tfrac12 \|\vW(s)\|^2 $ satisfies the
  differential equation
  \begin{equation}
    \delta_{ss} - \lambda(s)\delta_s - 2\alpha\delta = 
    \|\vW_s\|^2.
    \label{eq:length-of-W-ode}
  \end{equation}
  If $\vC$ is non-translating soliton, so that $\vv=0$, then $\delta(s) =
  \tfrac12\|\vC(s)\|^2$ satisfies
  \begin{equation}
    \delta_{ss} - \lambda(s)\delta_s - 2\alpha\delta = 1.
    \label{eq:length-of-C-ode}
  \end{equation}
\end{lemma}
\begin{proof}
  The general case implies the special case by choosing $\cB$ to be the orthogonal
  projection onto $R(\cA)$ for (\ref{eq:length-of-W-ode}) or by letting $\cB$ be
  the identity for (\ref{eq:length-of-C-ode}).  To prove the general case we
  simply compute the derivatives of $\delta=\frac12\|\cB\vC\|^2$,
  \begin{eqnarray*}
    \delta_s &=& \langle\cB\vC, \cB\vC_s\rangle, \\
    \delta_{ss} &=& \langle\cB\vC_s, \cB\vC_s\rangle
    + \langle\cB\vC, \cB\vC_{ss}\rangle\\
    &=&\|\cB\vC_s\|^2 + \bigl\langle \cB\vC, \cB\bigl\{ (\alpha+\cA)\vC
    + \vv +\lambda\vC_s\bigr\}
    \bigr\rangle\\
    &=&\|\cB\vC_s\|^2 + \alpha\|\cB\vC\|^2
    +\lambda \langle\cB\vC, \cB\vC_s\rangle\\
    &=&\|\cB\vC_s\|^2 + 2\alpha\delta + \lambda\delta_s.
  \end{eqnarray*}
\end{proof}

\begin{lemma} \label{lem:length-of-W-one-minimum}
  Let $\vC$ be a rotating soliton that is either translating ($\alpha=0$,
  $\cA\vv=0$, $\vv\neq0$) or not shrinking ($\alpha \geq 0$, $\vv=0$).  Let
  $\cB$ again be linear with $\cA\cB=\cB\cA$ and $\cB\vv=0$.  Assume moreover
  that $R(\cB)\subset R(\cA)$.

  If $\cB\vC(s)\neq0$ for some $s\in\R$, then $\|\cB\vC(s)\|$ is either strictly
  monotone on the entire soliton, or else it has exactly one minimum.

  The same statement applies to the length of $\vW$.
\end{lemma}
In the next section we will use this lemma with $\cB=\cA$.  In the case that has
the most natural interpretation $\cB$ is the orthogonal projection on $R(\cA)$,
so $\cB\vC=\vW$.

\begin{proof}
  Since $\cB\vC(s)\neq0$ for some $s\in\R$, we have $\delta(s)\neq0$ for at least
  some $s\in\R$, where $\delta(s) = \frac12\|\cB\vC(s)\|^2$.

  If $\delta_s\neq0$ for all $s$ then $\delta_s$ must have the same sign for all
  $s$, and therefore $\delta$ is monotone.

  The remaining case to consider is that $\delta_s(s_0)=0$ for some $s_0$.  We
  will show that this implies $\delta(s)$ has a strict local minimum at $s_0$, and
  in particular, that $\delta$ has no local maxima.  This implies that a strict
  local minimum must be unique, since any two strict local minima would have to
  bracket a local maximum.  Having established that $\delta$ has exactly one
  critical point $s_0$ it follows that $\delta(s)$ is decreasing for $s<s_0$ and
  increasing for $s>s_0$.

  Let $s_0$ be a critical point of $\delta$.  The differential equation for
  $\delta$ implies that
  \[
  \delta_{ss}(s_0) = 2\alpha\delta(s_0) + \|\cB\vC_s(s_0)\|^2\geq 0.
  \]
  If either $\alpha\delta(s_0)>0$ or $\cB\vC_s(s_0)\neq 0$ then we have
  $\delta_{ss}(s_0) >0$ so that $\delta$ has a local minimum at $s_0$.  We are
  therefore left with the possibility that both $\alpha\delta(s_0) = 0$ and
  $\cB\vC_s(s_0)=0$.

  If $\alpha>0$ then $\alpha\delta(s_0)=0$ implies $\delta(s_0)=0$ and therefore
  $\cB\vC(s_0)=0$.  Applying $\cB$ to both sides of the differential equation
  (\ref{eq:soliton-profile}) and using the assumption that $\cB$ commutes with
  $\cA$, we find that $\cB\vC$ satisfies a linear homogeneous differential
  equation,
  \begin{equation}
    \cB\vC_{ss} = (\alpha+\cA) \cB\vC + \lambda(s) \cB\vC_s.
    \label{eq:W-ode}
  \end{equation}
  It then follows from $\cB\vC(s_0) = \cB\vC_{s}(s_0)=0$ that $\cB\vC(s)=0$ for
  all $s$.  We had assumed this was not the case.

  Finally we have one more case left, namely, $\alpha=0$.  Compute the next two
  derivatives of $\delta$ using (\ref{eq:length-of-W-ode}):
  \begin{eqnarray*}
    \delta_{sss} &=& \lambda\delta_{ss} + \lambda_s\delta_{s}
    + 2\langle\cB\vC_s, \cB\vC_{ss}\rangle \\
    \delta_{ssss} &=&\lambda\delta_{sss} + 2\lambda_{s}\delta_{ss}
    +\lambda_{ss}\delta_s +2\|\cB\vC_{ss}\|^2 + 2\langle\cB\vC_s,
    \cB\vC_{sss}\rangle.
  \end{eqnarray*}
  At $s_0$ the conditions $\cB\vC_s=0$ and $\delta_s=\delta_{ss}=0$ imply first
  $\delta_{sss}=0$, and then $\delta_{ssss} = 2\|\cB\vC_{ss}\|^2$.

  If $\cB\vC_{ss}(s_0)=0$, then (\ref{eq:W-ode}) with $\alpha=0$ implies that
  $\cA\cB\vC(s_0) = \cB\vC_{ss}-\lambda\cB\vC_s = 0 $.  Since $\cA$ is invertible
  on $R(\cA)$ and since $R(\cB)\subset R(\cA)$ this forces $\cB\vC(s_0)=0$.  But
  then $\cB\vC(s_0) = \cB\vC_s(s_0) = 0$ so that (\ref{eq:W-ode}) implies
  $\cB\vC(s)=0$ for all $s$, once more contradicting our assumption.

  Since we have shown that $\cB\vC_{ss}(s_0)\neq 0$ we know that
  $\delta_{ssss}(s_0)>0$ and therefore $\delta$ has a strict local minimum at
  $s_0$ is this last case too.

\end{proof}

\begin{figure}[t]
  \begin{center}
    \includegraphics[width=0.75\textwidth]{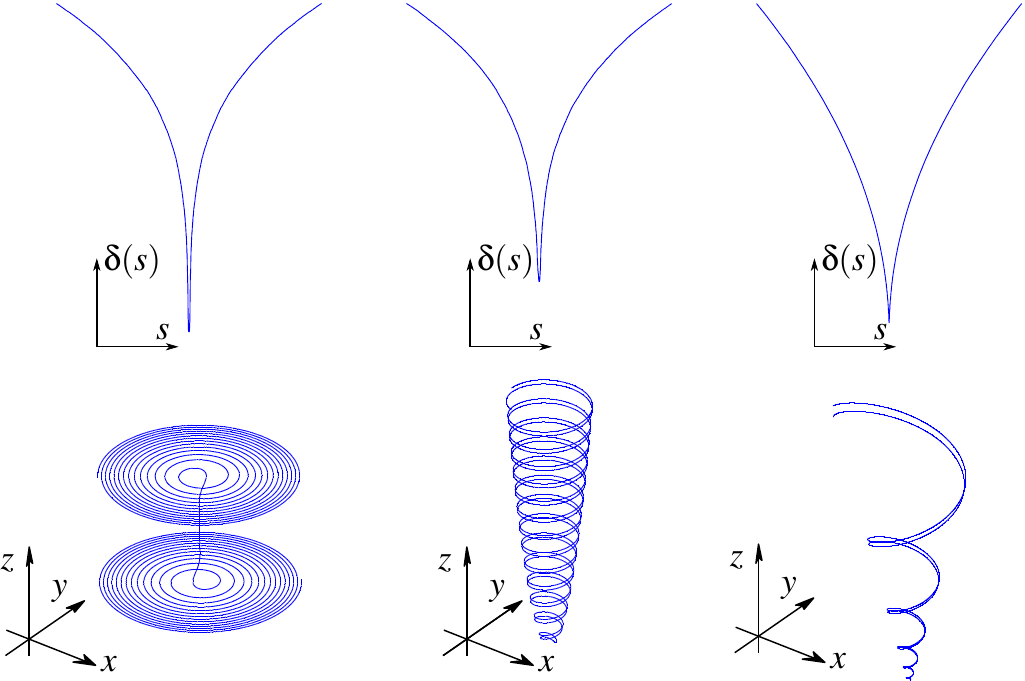}
  \end{center}
  \caption{Plot of the distance to the rotation axis for a number of
  solitons. }
  \label{fig:ZDistance}
\end{figure}

\subsection{Unbounded ends for non-shrinking solitons}
\label{sec:unbounded-ends-nonshrinking}

\begin{lemma}
  If $\vC$ is a rotating non-shrinking soliton, i.e.~a solution
  of~(\ref{eq:soliton-profile}) with $\vv=0$ and $\alpha\geq 0$, then both
  ends of $\vC$ are unbounded, i.e.
  \[
  \lim_{s\to \pm \infty}\|\vC(s)\| = \infty.
  \]
  \label{lem:non-shrinkers-unbounded}
\end{lemma}
\begin{proof}
  We have shown in the previous Lemma that $\|\vC(s)\|$ can have at
  most one minimum.  This implies that either $\lim_{s\to\infty}
  \|\vC(s)\| = \infty$ or else $\sup_{s\geq 0} \|\vC(s)\| <\infty$.
  We will show that the second alternative leads to a contradiction.

  Suppose $K=\limsup_{s\to\infty} \|\vC(s)\| < \infty$.  Choose a
  sequence $s_j\to\infty$ with $\|\vC(s_j)\| \to K$, and consider the
  solitons defined by
  \[
  \vC_j(s) \stackrel{\rm def}= \vC(s_j+s).
  \]
  On any bounded interval $|s|\leq L$ the $\vC_j$ are uniformly
  bounded.  Their derivatives $\vC_{j, s}$ are unit vectors and hence
  also uniformly bounded, and by virtue of the differential equation
  (\ref{eq:soliton-profile-w-pro}) the second derivatives are also
  uniformly bounded.  After passing to a subsequence we may therefore
  assume that the $\vC_j$ converge in $C^\infty_{\rm loc}$ to a
  solution $\vC_*$ of (\ref{eq:soliton-profile-w-pro}).  This limit
  satisfies  $\|\vC_*(s)\|\leq K$ for all $s\in \R$, and
  \[
  \|\vC_*(0)\| = \lim_{j\to\infty} \|\vC_j(s)\|= \lim_{j\to\infty} \|\vC(s_j)\| = K.
  \]
  Thus $\vC_*$ is a rotating-expanding soliton on which $\delta(s) =
  \tfrac12 \|\vC_*(s)\|^2$ attains a maximum at $s=0$.  By
  Lemma~\ref{lem:length-of-W-ode} we then have
  \[
  \delta_{ss}(0) = \lambda\delta_s(0) + \alpha\delta(0) + 1 = \alpha\delta(0)+1.
  \]
  Since $\alpha\geq 0$ we see that $\delta_{ss}(0)\geq 1>0$ so that
  $\delta(s)$ has a local minimum at $s=0$.  This is our contradiction.

  A slight modification of these arguments shows that the other end of
  the soliton is also unbounded, i.e.~that $\lim_{s\to -\infty}
  \|\vC(s)\| = \infty$.
\end{proof}

\begin{lemma}
  Let $\vC$ be a translating-rotating soliton ($\alpha=0$, $\vv\neq0$,
  $\cA\vv=0$).  Then $z(s) = \langle\vv, \vC(s)\rangle$ has at
  most one minimum, and
  \[
  \lim_{s\to \pm \infty} z(s) = \pm\infty.
  \]
  Furthermore, the $R(\cA)$ component $\vW(s)$ of $\vC(s)$ satisfies
  \begin{equation}
    \lim_{s\to \pm \infty}  \|\vW(s)\| = 0 \mbox{ or } \infty.
    \label{eq:W-to-zero-or-infty}
  \end{equation}
  \label{lem:ends-of-translating-rotating}
\end{lemma}
\begin{proof}
  By Lemma~\ref{lem:length-of-W-one-minimum} the quantity $\delta(s) =
  \frac12\|\vW(s)\|^2$ has at most one minimum.  It is therefore
  monotone for $s$ sufficiently large.

  If $\delta(s)$ remains bounded for $s\to\infty$, then
  \[
  \delta_\infty = \lim_{s\to\infty} \delta(s)
  \]
  exists.
  Consider the sequence of translates $\vC_j(s) = \vC(j+s) - \vZ(j)$
  ($j\in\N$) of our given soliton.  These curves are again
  solutions of (\ref{eq:soliton-profile}) (since $\alpha=0$ the
  equation is invariant under translations parallel to $N(\cA)$).
  Extracting a convergent subsequence one finds a soliton $\vC_*(s)$
  for which $\frac12\|\vW_*(s)\|^2 = \delta_\infty$ is constant. According to
  Lemma~\ref{lem:length-of-W-one-minimum} this is impossible unless
  $\vW_*(s) =0$ for all $s$.  So we see that for our original soliton
  we either have $\|\vW(s)\| \to\infty$ or $\|\vW(s)\|\to0$ as
  $s\to\infty$.

  The case $s\to -\infty$ follows by reversing the orientation of the
  curve and applying the same arguments.

  For $z(s)$ one has a differential equation,
  \begin{equation}
    z_{ss} = \lambda(s)z_s + \|\vv\|^2.
    \label{eq:rotate-translate-z-ode}
  \end{equation}
  This equation implies, as before, that $z$ is either strictly
  monotone or has a unique local minimum, which then also is a global
  minimum.  

  We also have the following trivial bound for $z_s$
  \[
  |z_s| = \left|\langle\vC_s, \vv\rangle\right| \leq \|\vv\|.
  \]

  \textbf{Case 1: assume that $z_s\geq 0$ for $s\geq s_0$. } 
  The coefficient $\lambda$ is given by (\ref{eq:lambda-revealed}),
  which in the current context is $\lambda = \langle\vC_s, \cA\vC + \vv\rangle$.
  As before we have
  \[
  \frac{d\lambda}{ds} = -\|\vC_{ss}\|^2 \leq 0.
  \]
  We therefore know that $\lambda$ is decreasing.  To show that
  $z(s)\to\infty$ we distinguish between a few cases, depending on the
  behaviour of $\lambda$ for large $s$.

  If $\lambda(s)>0$ for all $s\geq s_0$, then
  (\ref{eq:rotate-translate-z-ode}) implies $ z_{ss} \geq \lambda z_s$
  which, in turn, implies 
  \[
  \ln \frac{z_s(s)}{z_s(s_0)} = \int_0^s \frac{z_{ss}(s')}{z_s(s')} ds'
  \geq \int_{s_0}^s \lambda(s')ds'.
  \]
  Since $z_s\leq \|\vv\|$ for all $s$, we find that 
  \begin{equation}
    \int_{s_0}^\infty \lambda(s) ds \leq \ln \frac{\|\vv\|}{z_s({s_0})} <\infty.
    \label{eq:lambda-integrable-again}
  \end{equation}
  The integrating factor for (\ref{eq:rotate-translate-z-ode})
  is given by
  \begin{equation}
    m(s) = e^{-\int_{s_0}^s \lambda(s')ds'}.
    \label{eq:z-to-inf-proof-integrating-factor}
  \end{equation}
  In view of (\ref{eq:lambda-integrable-again}) we conclude that there exist
  $m_\pm$ such that 
  \[
  0<m_- \leq m(s) \leq m_+ \mbox{ for all }s\geq {s_0}.
  \]
  Integrating (\ref{eq:rotate-translate-z-ode}) we get
  \begin{eqnarray*}
    \bigl(m z_s\bigr)_s &= m(s)\|\vv\|^2 \geq m_-\|\vv\|^2 \\
    &\implies
    m(s) z_s(s) \geq  m_-\|\vv\|^2  s - C\\
    &\implies
    z_s(s) \geq  \frac{m_-}{m_+}\|\vv\|^2  s - C.
  \end{eqnarray*}
  Thus $z_s\to\infty$, which cannot be, since $z_s\leq \|\vv\|$.
  The contradiction shows that $\lambda$ cannot stay positive for all
  $s\ge s_0$.  Since $\lambda$ is strictly decreasing we must have
  $\lambda(s)<0$ for all large enough $s$.  We may assume that
  $\lambda(s) \leq -\delta$ is $s\geq s_1$ for certain $\delta$,
  $s_1$.

  Applying the variation of constants formula to
  (\ref{eq:rotate-translate-z-ode}) we find
  \begin{eqnarray*}
    z_s(s)
    &=&
    e^{\int_{s_1}^s\lambda(s')ds'}z_s({s_1})
    +
    \|\vv\|^2\int_{s_1}^s e^{\int_{s'}^s \lambda(s'')ds''}ds' \quad\mbox{(drop the
    1st term)} \\
    & \geq &
    \|\vv\|^2\int_{s_1}^s e^{\int_{s'}^s \lambda(s'')ds''}ds'\quad \mbox{(use
    $\lambda(s'')>\lambda(s)$)}\\
    & \geq &
    \|\vv\|^2\int_{s_1}^s e^{\lambda(s)(s-s')}ds'  \\
    & \geq & \|\vv\|^2 \; \frac{1-e^{\lambda(s)(s-s_1)}}{-\lambda(s)}. 
  \end{eqnarray*}
  The formula (\ref{eq:lambda-revealed}) for $\lambda$ implies
  \[
  |\lambda(s)| \leq \|\cA\vC(s)+\vv\| \leq C(1+s)
  \]
  since $\|\vC(s)\| \leq \|\vC(0)\| + |s|$.

  Thus we have $z_s\geq C/(1+s)$, which, upon integration, shows that
  $z(s)\to\infty$ as $s\to\infty$.  We are done with the case in which
  $z_s>0$ for $s\geq s_0$.

  \textbf{Case 2: assume that $z_s\leq 0$ for $s\geq s_0$. }
  If additionally we assume that $\lambda\leq 0$ for large enough $s$, then
  $z_{ss} = \lambda z_s + \|\vv\|^2 \geq\|\vv\|^2$.  Integrating, we find that
  $z_s$ would have to become positive for large enough $s$, which rules out the
  possibility that $\lambda(s)\leq0$ for some large enough $s$.

  We are left with the case that $\lambda\geq0$ for all $s\geq s_0$.  Since
  $\lambda$ is decreasing this implies $0\leq \lambda(s)\leq \lambda(s_0)$ for all
  $s\geq s_0$.

  Again consider the integrating factor $m(s)$ from
  (\ref{eq:z-to-inf-proof-integrating-factor}).
  We have for $s_0<s<s_1$
  \begin{equation}
    -m(s)z_s(s) = -m(s_1) z_s(s_1) + \|\vv\|^2 \int_{s}^{s_1} m(s') ds'.
    \label{eq:z-to-inf-proof-case-2-voc}
  \end{equation}
  All terms are positive so we can let $s_1\to\infty$ and conclude that
  \[
  \int_{s_0}^\infty m(s)ds<\infty.
  \]
  Since $m(s)$ is nonincreasing this implies that $\lim_{s\to\infty} m(s) = 0$,
  so that we can also conclude from (\ref{eq:z-to-inf-proof-case-2-voc}) that
  \[
  -z_s(s) = \|\vv\|^2 \int_{s}^{\infty} \frac{m(s')}{m(s)}ds'.
  \]
  Use $0\leq \lambda(s)\leq \lambda(s_0)$ to estimate
  \[
  \frac{m(s')}{m(s)} = e^{\int_{s}^{s'}\lambda(s'')ds''} \geq
  e^{-\lambda(s_0)(s'-s)},
  \]
  and from there
  \[
  -z_s(s) \geq \frac{\|\vv\|^2}{\lambda(s_0)}
  \]
  for all $s\geq s_0$.  After integration this implies that $\lim_{s\to\infty}
  z(s) = -\infty$.

\end{proof}

\subsection{Rotating Shrinking Solitons in $\R^n$ and Rotating
  Solitons on $\Sph^{n-1}$}
\label{sec:rotators-on-sphere}
We observed in \S\ref{sec:CS-Sn-1} that there is a one-to-one
correspondence between Curve Shortening on the sphere $\Sph^{n-1}$ and
certain shrinking solutions of Curve Shortening in $\R^n$.  Under this
correspondence some of the rotating shrinking solitons we study in
this paper correspond to purely rotating solitons on $\Sph^{n-1}$ as
studied by Hungerb\"uhler and Smoczyk in
\cite{2000HungerBuehlerSmoczyk} (e.g.~see Figure 6 in
\cite{2000HungerBuehlerSmoczyk}).  In particular, every rotating
soliton on $\Sph^{n-1}$ is also a shrinking rotating soliton for Curve
Shortening on $\R^n$.  The following Lemma shows directly that there
is a large class of rotating shrinking solitons that lie on a sphere.

\begin{lemma}
  Let $\alpha<0$ and $\vv=0$.  Then the rotating soliton that is
  tangent to the sphere with radius $(-\alpha)^{-1/2}$ is entirely
  contained in that sphere.
\end{lemma}
\begin{proof}
  The squared distance $\delta(s) = \frac12 \|\vC(s)\|^2$ to the
  origin satisfies \eqref{eq:length-of-C-ode}.  Uniqueness of
  solutions of differential equations implies that if
  $\delta(s_0)=1/(-2\alpha)$ and $\delta'(s_0) = 0$ holds for some
  $s_0$ then $\delta(s) = 1/(-2\alpha)$ for all $s\in\R$.
\end{proof}

\section{Behavior at infinity}
\label{sec:Behavior-at-infinity}

\subsection{High and low curvature decomposition of soliton ends}
In this section we consider the behaviour of solitons far away from the origin.
The main finding is that when $\vC$ (or rather its projection onto
$R(\alpha+\cA)$) is large, the soliton $\vC$ decomposes into two parts.  One
part consists of short arcs of high curvature on which $\vC$ is approximately
constant.  On such arcs the soliton is approximated by a solution of
\begin{equation}
  \vC_{ss} = \pro_{\vC_s}(\va)
  \label{eq:fast-flow}
\end{equation}
for some constant vector $\va$ (one has $\va \approx (\alpha+\cA)\vC+\vv$).  The
solutions to (\ref{eq:fast-flow}) with constant $\va$ are exactly the Grim
Reapers with translation velocity vector $\va$.

On the other part of the soliton the curvature $\vC_{ss}$ is small, and in fact
much smaller than the remaining terms in (\ref{eq:soliton-profile-w-pro}).  It
follows that on these low curvature arcs the vector $(\alpha+\cA)\vC + \vv$ is
nearly parallel to $\vC_s$.  Low curvature arcs are therefore approximately
solutions of
\begin{equation}
  \vC_s = \pm \ah(\vC)
  \label{eq:slow-flow}
\end{equation}
where
\begin{equation}
  \va(\vC) = (\alpha+\cA)\vC + \vv,
  \label{eq:a-ahat-def}
  \qquad \ah = \frac{\va} {\|\va\|}.
\end{equation}
After reparametrization (\ref{eq:slow-flow}) is equivalent to a linear equation.
Namely, if one introduces a parameter related to arc length via
\begin{equation}
  \sigma = \int \frac{ds}{\|\va(\vC(s))\|},
  \label{eq:sigma-def}
\end{equation}
then (\ref{eq:slow-flow}) is equivalent with $\vC_\sigma = \|\va\| \vC_s = \pm
\va(\vC)$, i.e.
\begin{equation}
  \pm \frac{d\vC}{d\sigma} = (\alpha+\cA) \vC + \vv.
  \label{eq:slow-flow-linear}
\end{equation}
When $\alpha\neq0$ we may assume that $\vv=0$, and the solutions to
(\ref{eq:slow-flow-linear}) are given by
\begin{equation}
  \vC = e^{\pm\sigma(\alpha+\cA)}\vC_0
  \label{eq:generalized-log-spiral}
\end{equation}
for arbitrary $\vC_0$.  These curves rotate according to $\cA$ and
simultaneously move toward or away from the origin at an exponential rate
$e^{\alpha\sigma}$.  In the simplest, two-dimensional case we get logarithmic
spirals; in higher-dimensional cases we get curves that spiral inward or outward
on cones.  Below we show that the low curvature parts of any soliton with
$\alpha\neq0$ are actually $\cO(1)$ close to one of the generalized logarithmic
spirals (\ref{eq:generalized-log-spiral}).

If the dilation parameter $\alpha=0$, then we must allow $\vv\neq 0$, although
we still have $\vv\in N(\cA)$.  In this case we can split $\vC$ into its
components in the null space and range of $\cA$,
\[
\vC = \vW+\vZ, \quad \vW\in R(\cA), \quad \vZ\in N(\cA).
\]
Then (\ref{eq:slow-flow-linear}) decouples into $ \vW_\sigma = \cA\vW$,
$\vZ_\sigma = \vv$, which shows that the general solution to
(\ref{eq:slow-flow-linear}) is
\begin{equation}
  \vC = e^{\sigma\cA}\vW_0 + \sigma \vv +\vZ_0
  \label{eq:generalized-helices}
\end{equation}
with $\vW_0\in R(\cA)$, $\vZ_0\in N(\cA)$ arbitrary constants.

The simplest case is again two dimensional with $N(\cA)=\{0\}$ and $\vv=0$.
Here the rotating solitons are the Yin-Yang curves that spiral outward.  On the
other hand, the corresponding solution (\ref{eq:generalized-helices}) to
(\ref{eq:slow-flow-linear}) parametrizes a circle; in particular, it remains
bounded.  This suggests that for $\alpha=0$ the solution of
(\ref{eq:generalized-helices}) can provide a good approximation of finite length
segments of the solitons, but that the approximation is not uniform as
$\sigma\to\pm\infty$.

\subsection{Small and distant Grim Reapers}
In what follows we will make the preceding discussion precise.  We
write the soliton equation (\ref{eq:soliton-profile-w-pro}) as
\begin{equation}
  \vT_s = \|\va\|\;\pro_{\vT}(\ah),
\end{equation}
and consider the quantities
\[
\mu = \langle\ah, \vT\rangle,
\qquad 
\nu = \|\va\|^4 \bigl(1-\mu^2\bigr).
\]
Here $\mu$ the cosine of the angle between the unit tangent $\vT$ and
$\ah(\vC)$.  If $\mu = \pm1$ then $\ah=\pm\vT$.  More generally, if
$\mu$ is close to $\pm1$ then $\ah$ almost lines up with $\vT$.  In
particular, we have the identities
\begin{equation}
  \|\ah-\vT\|^2 = 2(1-\mu),\qquad
  \|\ah+\vT\|^2 = 2(1+\mu).
  \label{eq:ah-pm-vT-and-mu}
\end{equation}
Since $\vC_{ss} = \va - \langle\va, \vT\rangle\vT = \|\va\| \bigl( \ah
-\mu\vT \bigr)$, 
and since $\vC_{ss}$ and $\vT$ are orthogonal, one finds that the
curvature of the soliton is given by
\begin{equation}
  \|\vC_{ss}\| = \|\va\|\;\sqrt{1-\mu^2}.
  \label{eq:and-the-curvature-is}
\end{equation}
Therefore
\[
\nu = \Bigl(\|\va\|\,\|\vC_{ss}\|\Bigr)^2
\]
can be interpreted as a scale-invariant curvature of the soliton.

We begin with an estimate for the rate at which $\mu$ and $\nu$ change
along a soliton.
\begin{lemma}
  \label{lem:diff-ineqs-for-mu-nu}
  There is a constant $C$, which only depends on $\alpha$ and $\cA$, such
  that $\mu$ satisfies
  \begin{equation}
    \|\va\|^{-1}\mu_s 
    \geq 1-\mu^2 - \frac{C} {\|\va\|^2} \sqrt{1-\mu^2}.
    \label{eq:mu-evolution}
  \end{equation}
  The quantity
  \[
  \nu = \|\va\|^4 (1-\mu^2)
  \]
  satisfies
  \begin{eqnarray}
    \label{eq:nu-evolution-bound-mupos}
    \|\va\|^{-1}\nu_s &\leq& -2\mu\nu + C\mu\sqrt{\nu} + \frac{C}
    {\|\va\|^2}\nu
    \mbox{ when $\mu>0$, and }  \\
    \label{eq:nu-evolution-bound-muneg}
    \|\va\|^{-1}\nu_s &\geq& -2\mu\nu - C\mu\sqrt{\nu} - \frac{C}
    {\|\va\|^2}\nu \mbox{ when $\mu<0$.}
  \end{eqnarray}
\end{lemma}
\begin{proof}
  We begin with
  \[
  \va_s =  (\alpha+\cA)\vT,\qquad
  \ah_s =  \|\va\|^{-1} \, \pro_{\ah}(\va_s) = \|\va\|^{-1} \,
  \pro_{\ah}\bigl((\alpha+\cA)\vT\bigr).
  \]
  This leads to
  \begin{eqnarray*}
    \mu_s
    &=& \langle\ah, \vT_s\rangle + \langle\ah_s, \vT\rangle \\
    &=& \|\va\|\langle\ah, \pro_{\vT}\ah\rangle +
    \|\va\|^{-1}\Bigl\langle \vT,
    \pro_{\ah}\bigl[(\alpha+\cA)\vT\bigr]
    \Bigr\rangle\\
    &=& \|\va\|(1-\mu^2)
    + \|\va\|^{-1}\bigl\langle\pro_\ah\vT, (\alpha+\cA)\vT\bigr\rangle\\
  \end{eqnarray*}
  Next, we observe that $\|\pro_\ah\vT\|^2 = 1-\mu^2$ so that
  \begin{eqnarray*}
    \|\va\|^{-1}\mu_s
    &\geq& 1-\mu^2 - \|\va\|^{-2} \|\alpha+\cA\| \sqrt{1-\mu^2}\\
    &=& 1-\mu^2 - \frac{C} {\|\va\|^2} \sqrt{1-\mu^2}.
  \end{eqnarray*}

  The quantity $\nu = \|\va\|^4(1-\mu^2)$ satisfies
  \[
  \|\va\|^{-1}\nu_s = -2\mu\|\va\|^3 \mu_s + 4\|\va\|^2\|\va\|_s
  (1-\mu^2)
  \]
  Using (\ref{eq:mu-evolution}) and
  \[
  \bigl|\; \|\va\|_s\bigr| \leq \|\va_s\| = \|(\alpha+\cA)\vT\| \leq
  \|\alpha+\cA\|
  \]
  we find that when $\mu>0$ one has
  \begin{eqnarray*}
    \|\va\|^{-1}\nu_s &\leq& -2\mu\|\va\|^4 \Bigl( 1-\mu^2 - \frac{C}
    {\|\va\|^2} \sqrt{1-\mu^2} \Bigr) + 4\|\va\|^2\|\alpha+\cA\|
    (1-\mu^2)\\
    &\leq& -2\mu\nu + C\mu\sqrt{\nu} + \frac{C} {\|\va\|^2}\nu,
  \end{eqnarray*}
  as claimed.  The inequality for $\mu<0$ follows in the same way.
\end{proof}

The following lemma tells us where the regions of large curvature are.
\begin{lemma}\label{lem:GR-here}
  There is a constant $C<\infty$ depending only on $\alpha$ and $\cA$
  such that for any point $\vC_0 = \vC(s_0)$ on a solution of
  (\ref{eq:soliton-profile-w-pro}) that satisfies
  \begin{equation}
    A_0 = \|\va(\vC_0)\| \geq C,
    \mbox{ and }
    |\mu(\vC_0, \vT_0)| \leq 1 - CA_0^{-4}
    \label{eq:T-a-angle-not-small}
  \end{equation}
  there exist $s_- < s_+$ with $s_0\in[s_-, s_+]$ and
  \begin{equation}
    s_+ - s_- \leq \frac{C}{A_0} \ln A_0,
    \label{eq:len-of-GR-bound}
  \end{equation}
  such that at the ends of the arc $\vC([s_-, s_+])$ one has
  \begin{equation}
    \mu(s_-) = - \bigl( 1 - CA_0^{-4}\bigr),\mbox{ and }
    \mu(s_+) = \bigl( 1 - CA_0^{-4}\bigr).
    \label{eq:mu-changes-sign-along-arc}
  \end{equation}
\end{lemma}
Since the quantity $\mu$ measures how well the unit tangent $\vT$
matches the vector $\va(\vC)$, this lemma says that if at some point
$\vC_0$ on a soliton the unit tangent is not very close to either
$\ah(\vC)$ or $-\ah(\vC)$, then the point $\vC_0$ is contained in an
arc on which the unit tangent $\vT$ turns from almost $-\ah(\vC)$ to
$+\ah(\vC)$.  The larger $A_0$ is, the smaller the upper bound
(\ref{eq:len-of-GR-bound}) for the length of the arc; if the arc is
very short then $\vC\approx\vC_0$ on the arc, and it approximately
satisfies (\ref{eq:fast-flow}) with $\va = \va(\vC_0)$, i.e.~the arc
is approximately a Grim Reaper.  See Figure~\ref{fig:GR-in-the-distance}.

\begin{figure}[t]
  \centering
  \includegraphics{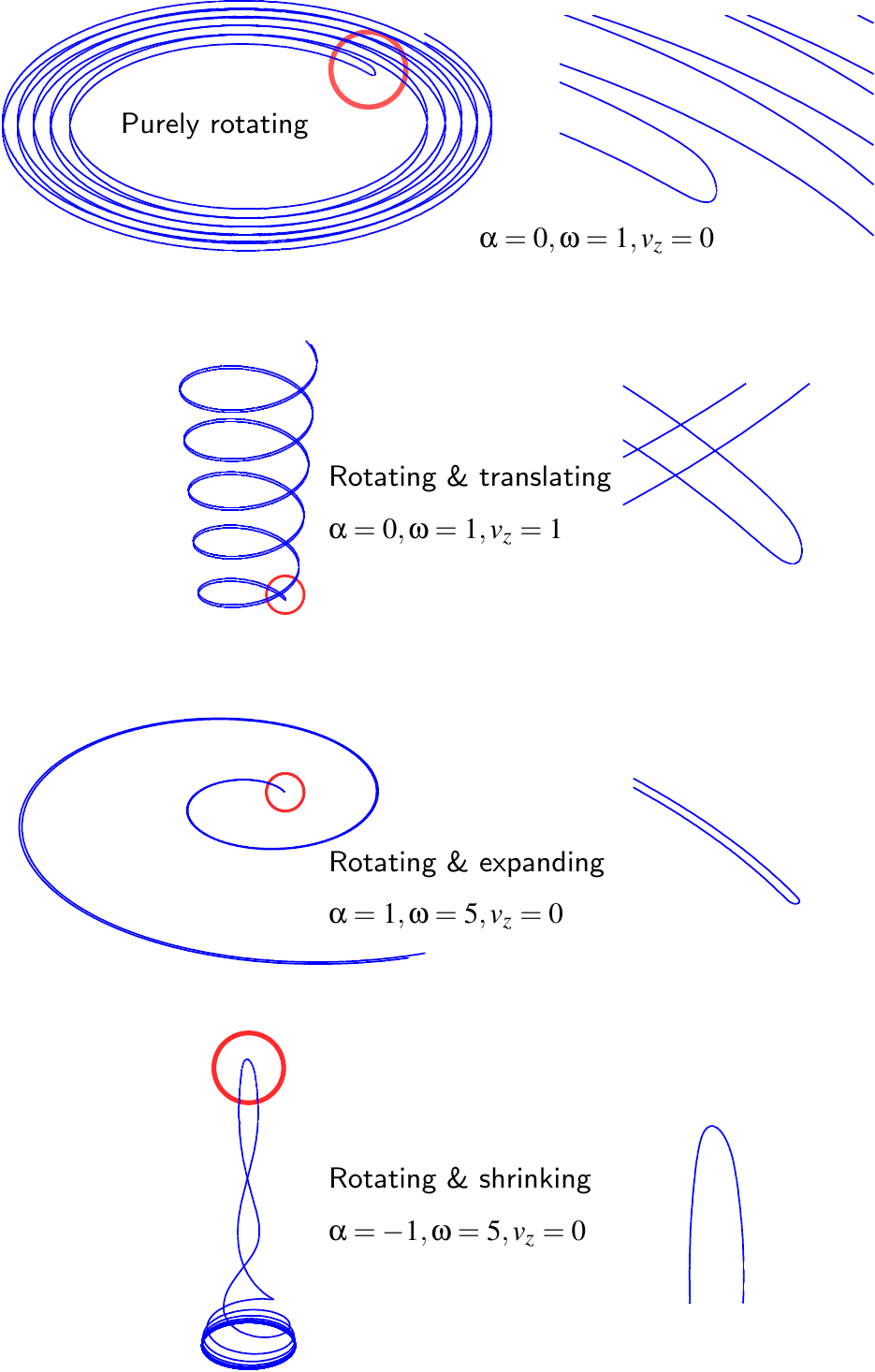}
  \caption{High curvature segments of solitons approximate Grim
    Reapers (Lemma~\ref{lem:GR-here}.)  The solitons here rotate
    around the $z$-axis with rate $\omega$, and translate along the
    $z$-axis with velocity $v_z$.  Magnifications of the high
    curvature segments of the solitons are shown in the right
    column. }
  \label{fig:GR-in-the-distance}
\end{figure}

\begin{proof}
  Consider the arc of length $2$ centered at $\vC_0$.  Along this arc
  we have $\|\vC-\vC_0\| \leq 1$, and hence
  \[
  \Bigl|\|\va\| - \|\va_0\|\Bigr| \leq \|\va - \va_0\| \leq
  \|\alpha+\cA\| \; \|\vC-\vC_0\| \leq C.
  \]
  If we assume that $A_0 \geq 2C$ then this implies that
  \[
  \|\va\| \geq A_0 - C \geq \tfrac12 A_0\mbox{ and } \|\va\| \leq A_0
  + C \leq 2A_0.
  \]
  Along this arc it follows from (\ref{eq:mu-evolution}) that
  \begin{equation}
    \|\va\|^{-1} \mu_s \geq 1-\mu^2 - \frac{4C}{A_0^2} \sqrt{1-\mu^2}
    \geq \tfrac12 (1-\mu^2),
    \label{eq:mu-growth-lowerbound}
  \end{equation}
  provided $\sqrt{1-\mu^2} \geq 8C/A_0^2$, i.e.\ provided
  \[
  |\mu| \leq 1- CA_0^{-4}.
  \]
  Integrating the differential inequality
  (\ref{eq:mu-growth-lowerbound}) we find that the length of any arc
  containing $\vC_0$ on which the above inequality holds is at most
  \[
  \makebox[2em]{$\displaystyle%
  \int\limits_{\mu = -1+CA_0^{-4}} ^{1-CA_0^{-4}}$}ds \leq
  \makebox[4em]{$\displaystyle%
  \int\limits_{-1+CA_0^{-4}} ^{1-CA_0^{-4}}$}
  \frac{2}{\|\va\|}\frac{d\mu}{1-\mu^2} \leq \frac{4}{A_0}
  \int\limits_{-1+CA_0^{-4}} ^{1-CA_0^{-4}} \frac{d\mu}{1-\mu^2} \leq
  \frac{C}{A_0}\ln A_0.
  \]
  If $A_0$ is sufficiently large then $CA_0^{-1}\ln A_0 <1$, and the
  arc is contained within the arc of length 2 with which we began the
  proof.
\end{proof}

\begin{lemma}
  \label{lem:growth-of-a}
  If $\|\va\|\geq C$ and $\mu\geq 1-C\|\va\|^{-4}$, then
  \begin{equation}
    \left|\frac{d\|\va\|}{ds} - \alpha\right| \leq C\|\va\|^{-2}
    \label{eq:a-growth-bound-case-mu-pos}
  \end{equation}
  If $\|\va\|\geq C$ and $\mu\leq -1+C\|\va\|^{-4}$, then
  \begin{equation}
    \left|\frac{d\|\va\|}{ds} + \alpha\right| \leq C\|\va\|^{-2}.
    \label{eq:a-growth-bound-case-mu-neg}
  \end{equation}
\end{lemma}
\begin{proof}
  From $\va = (\alpha+\cA)\va + \vv$ we get
  \[
  \frac{d}{ds}\|\va\|
  = \bigl\langle \ah, \va_s\bigr\rangle 
  = \bigl\langle \ah, (\alpha+\cA)\vT \bigr\rangle .
  \]
  If $\mu \geq 1- C\|\va\|^{-4}$, then (\ref{eq:ah-pm-vT-and-mu})
  implies
  \[
  \|\ah-\vT\| = \sqrt{2(1-\mu)} \leq C\|\va\|^{-2},
  \]
  so that
  \begin{eqnarray*}
    \frac{d}{ds}\|\va\| &=&  \bigl\langle \vT, (\alpha+\cA)\vT\bigr\rangle  +
    \bigl\langle \ah-\vT, (\alpha+\cA)\vT\bigr\rangle \\
    &=&  \alpha + \bigl\langle \ah-\vT, (\alpha+\cA)\vT\bigr\rangle 
  \end{eqnarray*}
  where we have used that $(\vT, \cA\vT) = 0$ and $(\vT,\vT)=1$.
  Hence
  \[
  \left|\frac{d\|\va\|}{ds} - \alpha\right| \leq \left|\bigl\langle \ah-\vT,
  (\alpha+\cA)\vT\bigr\rangle \right| \leq \|\ah-\vT\| \; \|\alpha+\cA\|\;
  \|\vT\| \leq C\|\va\|^{-2}.
  \]
  Thus (\ref{eq:a-growth-bound-case-mu-pos}) is proved.  To prove
  (\ref{eq:a-growth-bound-case-mu-neg}) we assume $\mu\leq
  -1+C\|\va\|^{-4}$ and use (\ref{eq:ah-pm-vT-and-mu}) again to
  conclude that $\|\ah+\vT\| \leq C\|\va\|^{-2}$.  From there we get
  \[
  \left|\frac{d\|\va\|}{ds} + \alpha\right| \leq \left|\bigl\langle \ah+\vT,
  (\alpha+\cA)\vT\bigr\rangle \right| \leq \|\ah+\vT\| \; \|\alpha+\cA\|\;
  \|\vT\| \leq C\|\va\|^{-2}.
  \]
\end{proof}

\begin{lemma}\label{lem:a-monotone}
  There is a constant $C$ depending only on $\alpha$ and $\cA$ such
  that for $\alpha>0$ it follows from $\|\va\|\geq C$ that $\|\va\|$
  is increasing when $\mu\geq 1- C\|\va\|^{-4}$ and decreasing when
  $\mu\leq -1 + C\|\va\|^{-4}$.

  On the other hand, if $\alpha<0$ and $\|\va\|\geq C$ then $\|\va\|$
  is decreasing when $\mu\geq 1- C\|\va\|^{-4}$ and increasing when
  $\mu\leq -1 + C\|\va\|^{-4}$.
\end{lemma}
\begin{proof}
  In the case $\alpha>0$ we let $C_1$ be the constant from
  (\ref{eq:a-growth-bound-case-mu-pos}).  Then if $\|\va\| \geq
  \sqrt{\alpha/2C_1}$ then (\ref{eq:a-growth-bound-case-mu-pos})
  implies that $d\|\va\|/ds \geq \tfrac12\alpha>0$.  So the Lemma
  holds if we set $C=\sqrt{\alpha/2C_1}$.

  The case $\alpha<0$ is entirely analogous.
\end{proof}

Consider the two regions
\begin{eqnarray*}
  \cR_+(K) &=&  \bigl\{ (\vC, \vT) : \|\va\| \geq K, \nu\leq K, \mu>0\bigr\}\\
  \cR_-(K) &=&  \bigl\{ (\vC, \vT) : \|\va\| \geq K, \nu\leq K,
  \mu<0\bigr\}
\end{eqnarray*}
where $\va = (\alpha + \cA)\vC$.

\begin{lemma}\label{lem:a-C-equivalent}
  For any $\vC\in\R^n$ one has 
  \[
  |\alpha|\;\|\vC\| \leq \|(\alpha+\cA)\vC\| \leq \|\alpha+\cA\|\;
  \|\vC\|.
  \]
\end{lemma}
\begin{proof}
  The second inequality is just the definition of the operator norm
  $\|\alpha+\cA\|$.  The first inequality follows from the fact that
  $\cA$ is antisymmetric.  By Cauchy's inequality and $(\vC,
  \cA\vC)=0$ one has
  \[
  \pm\alpha \|\vC\|^2 = \pm\bigl\langle\vC, (\alpha+\cA)\vC\bigr\rangle \leq
  \|\vC\|\;\|(\alpha+\cA)\vC\|,
  \]
  which implies the Lemma.
\end{proof}

\begin{lemma}
  If $K$ is large enough then for any solution $(\vC(s), \vT(s))$ to
  (\ref{eq:soliton-profile-as-system}) with $\|\va(0)\|\geq 2K$ one
  has
  \begin{itemize}
  \item either $\bigl(\vC(0), \vT(0)\bigr) \in \cR_+(K)\cup \cR_-(K)$,
  \item or there exist $s_-<0<s_+$ such that $(\vC(s_+),
    \vT(s_+))\in\cR_+(K)$ and $(\vC(s_-), \vT(s_-))\in\cR_-(K)$.
    Moreover, $s_\pm$ are bounded by
    \[
    |s_\pm| \leq C
    \|\va(0)\|^{-1}\ln\|\va(0)\|.
    \]
  \end{itemize}
\end{lemma}
\begin{proof}
  This is a reformulation of Lemma~\ref{lem:GR-here}.
\end{proof}

\begin{figure}
  \centering
  \includegraphics[width=\textwidth]{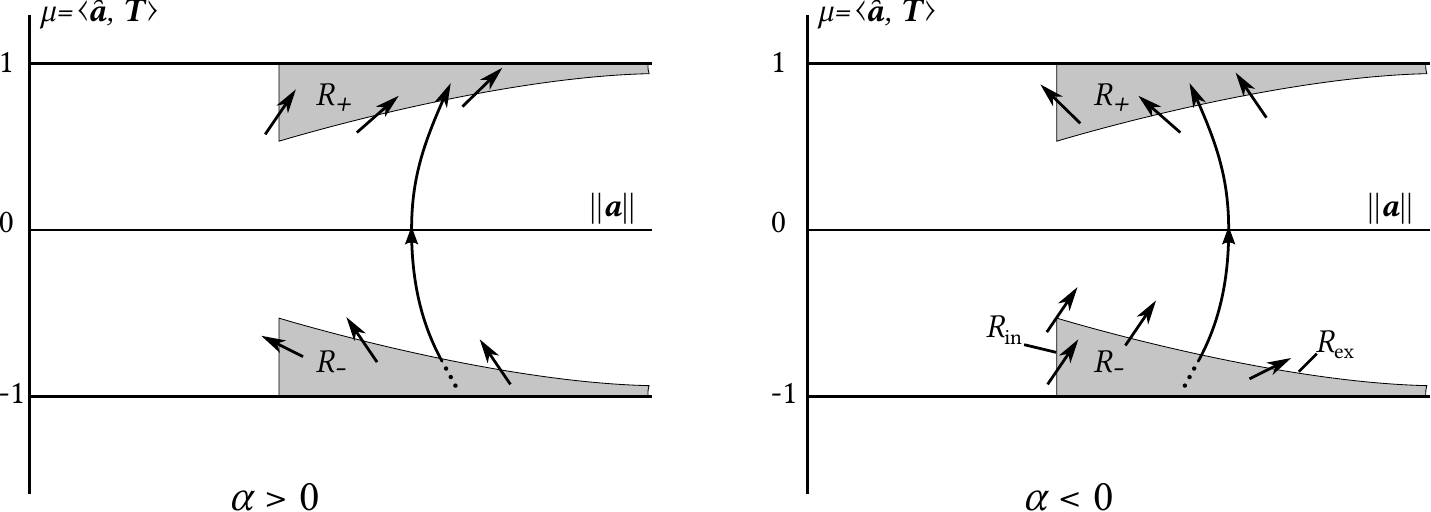}
  \caption{A caricature of the soliton flow
  (\ref{eq:soliton-profile-as-system}) in the cases $\alpha>0$ and
  $\alpha<0$.  }
\end{figure}

\begin{lemma} \label{lem:alpha-pos-R-invariant} If $\alpha>0$ then for
  sufficiently large $K$ the region $\cR_+(K)$ is forward invariant
  and the region $\cR_-(K)$ is backward invariant under the soliton
  flow (\ref{eq:soliton-profile-as-system}) on $\R^n\times \Sph^{n-1}$.

  If $\bigl(\vC(s), \vT(s)\bigr) \in \cR_+(C)$ for all $s\geq 0$, then
  \begin{equation}
    \vC = e^{\sigma(\alpha+\cA)}\vGa + \cO(e^{-\alpha\sigma})
    \label{eq:asymptotic-log-spiral}
  \end{equation}
  for some constant vector $\vGa\neq0$, provided $\sigma$ is defined
  as in~(\ref{eq:sigma-def}).

  If $\alpha<0$, then there is a $K<\infty$ such that any solution
  $\bigl(\vC(s), \vT(s)\bigr)$ of~(\ref{eq:soliton-profile-as-system}) that
  remains in $\cR_-(K)$ for all $s\geq 0$ satisfies
  \begin{equation}
    \vC = e^{-\sigma(\alpha+\cA)}\vGa + \cO(e^{\alpha\sigma})
    \label{eq:asymptotic-log-spiral-alpha-neg}
  \end{equation}
  for some constant vector $\vGa\neq0$, again, provided $\sigma$ is defined
  as in~(\ref{eq:sigma-def}).
\end{lemma}
\begin{proof}
  In the region $\cR_+(K)$ we have $\|\va\| \geq K$, and hence, by
  Lemma~\ref{lem:a-C-equivalent}, $\|\vC\|\geq CK$ in $\cR_+(K)$.

  We also have $\nu\leq K$ in $\cR_+(K)$.  By definition of $\nu$ this implies
  $\mu \geq 1- CK/\|\va\|^4$.  According to Lemma~\ref{lem:a-monotone} we find
  that for sufficiently large $K$ one has
  \begin{equation}
    \frac{d\|\va\|}{ds} \geq \frac{\alpha}{2}
    \label{eq:a-keeps-growing}
  \end{equation}
  in $\cR_+(K)$.  Thus the only way that a solution could escape from $\cR_+(K)$
  is by violating the inequality $\nu\leq K$.  However, in $\cR_+(K)$ we have,
  by Lemma~\ref{lem:diff-ineqs-for-mu-nu},
  \begin{eqnarray*}
    \|\va\|^{-1} \frac{d\nu}{ds} &\leq&
    \bigl(-2\mu + C\|\va\|^{-2}\bigr)\nu + C\sqrt\nu\\
    &\leq& \bigl(-2+ C\|\va\|^{-4} + C\|\va\|^{-2}\bigr)\nu + C\sqrt\nu
  \end{eqnarray*}
  If at any point on a solution in $\cR_+(K)$ one has $\nu=K$, then one also has
  \[
  \|\va\|^{-1} \frac{d\nu}{ds} \leq \bigl(-2+CK^{-2}\bigr)K + C\sqrt K.
  \]
  If $K$ is large enough this implies $\frac{d\nu}{ds} < 0$ when $\nu=K$.
  Therefore $\cR_+(K)$ is forward invariant when $K$ is sufficiently large.

  The exact same arguments show that $\cR_-(K)$ is backward invariant,
  especially if one observes that the transformation $(\vC, \vT) \mapsto
  (\vC, -\vT)$ maps $\cR_+(K)$ to $\cR_-(K)$ and  reverses the orientation
  of the soliton flow.  Therefore forward invariance of $\cR_+(K)$ is
  equivalent with backward invariance of $\cR_-(K)$.

  We now turn to the proof of (\ref{eq:asymptotic-log-spiral}).  If a soliton
  satisfies $\bigl(\vC(s), \vT(s)\bigr) \in \cR_+(K)$ for all $s\geq 0$, then we
  have $\nu\leq K$ and hence $\mu \geq 1-C\|\va\|^{-4}$ for all $s\geq 0$.  We
  are assuming that $\alpha>0$, so we have
  \[
  \|\va\| = \|(\alpha + \cA)\vC\| \geq \alpha \|\vC\|,
  \]
  and therefore $\mu \geq 1- C\|\vC\|^{-4}$.  Since $\|\ah-\vT\|^2 = 2(1-\mu)$
  we also have
  \begin{equation}
    \|\vT - \ah \| \leq C\|\vC\|^{-2}.
    \label{eq:T-a-deviation-bound}  
  \end{equation}
  Recall that $\ah = \va/\|\va\|$ and $\va = (\alpha+\cA)\vC$ (since $\alpha>0$
  we have $\vv=0$).  We get
  \[
  \left\| \|(\alpha+\cA)\vC\| \frac{d\vC} {ds} - (\alpha+\cA)\vC \right\| \leq
  \|(\alpha+\cA)\vC\| \frac{C}{ \|\vC\|^{2}} \leq \frac{C}{\|\vC\|}
  \]
  Defining $\sigma$ as in (\ref{eq:sigma-def}), we arrive at
  \begin{equation}
    \frac{d\vC} {d\sigma} - (\alpha+\cA)\vC = \vg(\sigma),
    \mbox{ where }
    \|\vg(\sigma)\| \leq \frac{C} {\|\vC\|}.
    \label{eq:spiral-diffeq-for-vC}  
  \end{equation}
  In particular, $\|\vg(\sigma)\|$ is bounded since $\|\vC\|$ is bounded away
  from $0$ in $\cR_+(K)$.

  Multiply (\ref{eq:spiral-diffeq-for-vC}) with the integrating factor
  $e^{\sigma(\alpha+\cA)}$ and integrate:
  \begin{equation}
    e^{-\sigma_*(\alpha+\cA)}\vC(\sigma_*)
    =
    e^{-\sigma(\alpha+\cA)}\vC(\sigma)
    +
    \int_\sigma^{\sigma_*}
    e^{-\sigma'(\alpha+\cA)}\vg(\sigma') \;d\sigma'
    \label{eq:spiral-diffeq-integrated}
  \end{equation}
  Since $\cA$ is antisymmetric $e^{\sigma\cA}$ is orthogonal and
  $\|e^{\sigma(\alpha+\cA)}\| = e^{\sigma\alpha}$. Therefore boundedness of
  $\vg(\sigma)$ implies that the integrand in
  (\ref{eq:spiral-diffeq-integrated}) decays exponentially.  This allows us
  to take the limit $\sigma_*\to\infty$ and conclude that
  \begin{equation}
    \vGa = \lim_{\sigma_*\to\infty} e^{-\sigma_*(\alpha+\cA)}\vC(\sigma_*)
    \label{eq:Gamm-def}
  \end{equation}
  exists.  Moreover, we have
  \[
  \vGa = e^{-\sigma(\alpha+\cA)}\vC(\sigma) + \int_\sigma^{\infty}
  e^{-\sigma'(\alpha+\cA)}\vg(\sigma') \;d\sigma'
  \]
  and therefore
  \begin{equation}
    \vC(\sigma) =
    e^{\sigma(\alpha+\cA)} \vGa
    - \int_\sigma^\infty  e^{-(\sigma'-\sigma)(\alpha+\cA)} \vg(\sigma')\;d\sigma'.
    \label{eq:spiral-asymptotics}
  \end{equation}
  This last integral can be estimated by
  \begin{eqnarray*}
    \left\|
    \int_\sigma^\infty e^{-(\sigma'-\sigma)(\alpha+\cA)} \vg(\sigma')\;d\sigma'
    \right\|
    &\leq& \sup_{\sigma'\geq \sigma}\|\vg(\sigma')\|
    \int_{\sigma}^\infty e^{-\alpha(\sigma'-\sigma)}d\sigma' \\
    &=& \frac{1}{\alpha}\sup_{\sigma'\geq\sigma}\|\vg(\sigma')\|.
  \end{eqnarray*}
  In particular, the integral is bounded.  It follows that $\vGa\neq0$, for if
  $\vGa$ were to vanish then (\ref{eq:spiral-asymptotics}) would imply that
  $\vC(\sigma)$ is bounded for all $\sigma\geq0$.  This is impossible when
  $\alpha>0$ because (\ref{eq:a-keeps-growing}) says that $\|\va\|$ must keep
  growing as long as the solution stays in $\cR_+(K)$.

  Having established (\ref{eq:spiral-asymptotics}) with $\vGa\neq0$ we can
  conclude that $\|\vC\| \geq C e^{\alpha\sigma}$.  Combined with $\|\vg\|\leq
  C/\|\vC\|$ this implies $\|\vg(\sigma)\| \leq Ce^{-\alpha\sigma}$.  Our
  estimate for the integral in (\ref{eq:spiral-asymptotics}) then implies that
  this integral is $\cO(e^{-\alpha\sigma})$, as claimed in
  (\ref{eq:asymptotic-log-spiral}).

  To prove (\ref{eq:asymptotic-log-spiral-alpha-neg}) one can use very similar
  arguments.  Since $\bigl(\vC(s), \vT(s)\bigr) \in \cR_-(K)$ for
  all $s>0$ we have $\mu < -1+C\|\vC\|^{-4}$, and hence $ \|\ah + \vT\| \leq
  C\|\vC\|^{-2} $.   From there one concludes that 
  \[
  \frac{d\vC}{d\sigma} = -(\alpha+\cA)\vC + \cO(\|\vC\|^{-1}),
  \]
  which then leads to (\ref{eq:asymptotic-log-spiral-alpha-neg}).
\end{proof}

When $\alpha<0$ the regions $\cR_\pm(K)$ are no longer invariant for the soliton
flow in either forward or backward direction, and most orbits that enter
$\cR_-(K)$ will leave $\cR_-(K)$ after a while.  However, there do exist
solutions of the soliton flow that, if you follow them in the direction of
$\vT$, remain in $\cR_-(K)$ forever.  We have just shown in
Lemma~\ref{lem:alpha-pos-R-invariant} that such orbits are again asymptotic to
logarithmic spirals.  The following Lemma constructs these special solitons via
a topological ``shooting method.''

\begin{lemma}
  \label{lem:alpha-neg-R-isolating-block}
  Let $\alpha<0$.
  If $\|\va_0\|\geq K$ then there is a $\vT_0\in \Sph^{n-1}$ with
  $\|\vT_0 + \ah(\vC_0)\|\leq CK\|\va_0\|^{-4}$ such that the solution
  $(\vC(s), \vT(s))$ of (\ref{eq:soliton-profile-as-system}) with
  $(\vC(0), \vT(0)) = (\vC_0, \vT_0)$ remains in $\cR_-(K)$ for all
  $s\geq 0$.
\end{lemma}

\begin{proof}
  The map $\phi : (\vC, \vT)\mapsto (\va, \vT) = ((\alpha+\cA)\vC,
  \vT)$ is a homeomorphism of the set $\cR_-(K)$ with
  \[
  \phi \bigl[\cR_-(K)\bigr] = \bigl\{(\va, \vT) : \|\va\| \geq K,\;
  \bigl\langle \ah, \vT \bigr\rangle \leq
  -\sqrt{1-K\|\va\|^{-4}}\bigr\}.
  \]
  The boundary of $\cR_-$ consists of two parts,
  \[
  \Rin = \phi^{-1}\bigl\{(\va, \vT) : \|\va\| = K,\; \bigl\langle \ah,
  \vT \bigr\rangle \leq -\sqrt{1-K\|\va\|^{-4}}\bigr\}
  \]
  and
  \[
  \Rex = \phi^{-1}\bigl\{(\va, \vT) : \|\va\| \geq K,\; \bigl\langle
  \ah, \vT \bigr\rangle = -\sqrt{1-K\|\va\|^{-4}}\bigr\}
  \]

  As above one concludes that within $\cR_-(K)$ one has
  \[
  \frac{d} {ds}\|\va\| \geq -\alpha/2 >0.
  \]
  Therefore any solution that starts in $\cR_-(K)$ can never reach
  $\Rin$ again, since $\|\va\| = k$ on $\Rin$; if a solution starts in
  $\cR_-(K)$ and then exits $\cR_-(K)$, it must hit $\Rex$.

  The same arguments which used Lemma~\ref{lem:diff-ineqs-for-mu-nu}
  to establish forward invariance of $\cR_+(K)$ when $\alpha>0$ show
  in the case $\alpha<0$ that, when a solution within $\cR_-(K)$
  reaches $\nu=K$, one has $\nu_s>0$.  Hence solutions that start in
  and then exit $\cR_-(K)$ cross $\Rex$ transversally when they exit.
  It follows that for those initial data $(\vC_0, \vT_0) \in \cR_-(K)$
  whose orbit $\{(\vC(s), \vT(s)) : s\geq 0\}$ under the soliton flow
  eventually exits $\cR_-(K)$, the point $E(\vC_0, \vT_0) \in \Rex$ at
  which they exit is a continuous (and even smooth) function of the
  initial point $(\vC_0, \vT_0)$.  We call the map $E:(\vC_0,
  \vT_0)\mapsto E(\vC_0, \vT_0)$ the \emph{exit map.}

  We now consider the topology of the exit set $\Rex$.

  For any fixed $\va$ with $\|\va\|\geq K$ the set of unit vectors
  $\vT\in \Sph^{n-1}\subset\R^n$ that satisfy
  \[
  \bigl\langle \ah, \vT \bigr\rangle = -\sqrt{1-K\|\va\|^{-4}}
  \]
  is an $(n-2)$-dimensional sphere in a hyperplane in $\R^n$
  perpendicular to $\va$.  From this one sees that the map
  \begin{equation}
    (\vC, \vT) \mapsto
    \Bigl(\ah, \frac{\pro_{\ah}\vT} {\|\pro_\ah\vT\|}\Bigr)
    \label{eq:homotopy-equivalence}
  \end{equation}
  is a homotopy equivalence of $\cR_-(K)$ with the unit tangent bundle
  $T_1\Sph^{n-1}$ of $\Sph^{n-1}$.  Under this homotopy equivalence the set
  of initial data $(\vC_0, \vT_0)$ with prescribed $\vC_0$ maps to the fiber
  over $\ah_0$.

  It is a fact from homotopy theory that the fiber in the unit tangent bundle
  $T_1\Sph^{n-1}\to \Sph^{n-1}$ is not contractible in $T_1\Sph^{n-1}$.
  Indeed, the fiber bundle $\Sph^{n-2} \to T_1\Sph^{n-1} \to
  \Sph^{n-1}$ leads to a long exact sequence of homotopy groups
  \[\fl
  \cdots\to \pi_{n-1}(\Sph^{n-1}) \stackrel{\partial_*}\longrightarrow
  \pi_{n-2}(\Sph^{n-2}) \longrightarrow
  \pi_{n-2} (T_1\Sph^{n-1}) \longrightarrow
  \pi_{n-2}(\Sph^{n-1}) \to\cdots
  \]
  Since $\pi_{n-2}(\Sph^{n-1})$ is trivial $\pi_{n-2} (T_1\Sph^{n-1})$ is
  isomorphic to $\pi_{n-2}(\Sph^{n-2}) / \partial_*\pi_{n-1}(\Sph^{n-1})
  $.  Both $\pi_{n-2}(\Sph^{n-2}) $ and $\pi_{n-1}(\Sph^{n-1}) $ are
  isomorphic to $\Z$, and by explicit computation one finds that
  $\partial_*$ vanishes when $n$ is odd, and acts by multiplication
  with 2 when $n$ is even.  Thus for even $n$ we find that
  $\pi_{n-2} (T_1\Sph^{n-1}) \simeq \Z_2$, while for odd $n$ we have $
  \pi_{n-2} (T_1\Sph^{n-1})\simeq \Z$.
  In either case, the fiber $\Sph^{n-2} $ represents a homotopy class
  that generates $\pi_{n-2}(\Sph^{n-2})$, and does not lie in the image
  of $\partial_*$.  Its image in $\pi_{n-2} (T_1\Sph^{n-1}) $ is
  therefore nontrivial, and so the fiber is not contractible in
  $T_1\Sph^{n-1}$.

  To prove the Lemma we consider the given set of initial data: we are
  given one point $\vC_0$, and we let $\vT$ be any unit vector that
  satisfies
  \begin{equation}
    \bigl\langle \ah_0, \vT \bigr\rangle \leq -\sqrt{1-K\|\va_0\|^{-4}} 
    \label{eq:initial-cap}
  \end{equation}
  This set of points forms an $(n-1)$-dimensional ball $\B^{n-1}\subset
  \cR_-(K)$.  If the orbit starting at all possible initial data $\vT_0$ in this
  ball $\B^{n-1}$ were to exit $\cR_-(K)$, then the exit map $E$
  restricted to $\B^{n-1}$ would provide a continuous map $\B^{n-1} \to
  \Rex$.  On the other hand, the points on the boundary of $\B^{n-1}$
  are already in the exit set $\Rex$, so the exit map restricted to
  $\partial \B^{n-1}$ is the identity.  Under the homotopy
  equivalence~(\ref{eq:homotopy-equivalence}) the boundary $\partial
  \B^{n-1}$ maps to a fiber in the unit tangent bundle $T_1 \Sph^{n-2}$.
  The exit map restricted to $\B^{n-1}$ would then provide a
  contraction of this fiber within $T_1\Sph^{n-2}$ to a point, which is
  impossible.  We therefore conclude that for at least one choice of
  $\vT_0$ satisfying (\ref{eq:initial-cap}) the orbit of the soliton
  flow will not leave $\cR_-(K)$.

\end{proof}

\section{Global behavior of Rotating-Dilating Solitons}
\label{sec:global-behaviour}
\subsection{Decomposition of $\R^n$ into complex subspaces determined
by $\cA$}\label{sec:complex-subspaces}
In the previous section we have seen that rotating-dilating solitons can have
ends that are asymptotic to logarithmic spirals.  Here we determine all possible
asymptotic behaviours that an end of a rotating-dilating soliton can have.

There are a number of trivial rotating-dilating solitons, namely, circles with
appropriate radius, and Abresch-Langer curves in case $N(\cA)$ is nontrivial.
We will show that all bounded ends of rotating-dilating solitons must be
asymptotic to these simpler solitons.

It will be convenient to change our notation for the normal form of the
infinitesimal rotation matrix $\cA$.  Recall that in a suitable basis $\cA$ is
of the form (\ref{eq:rotation-normal-form}).  In the most general case several
of the rotation frequencies $\omega_j$ may coincide.  We now denote the set of
\textit{distinct} nonzero eigenvalues of $i\cA$ by $\left\{ \pm\Omega_k : 1\leq
k\leq \bar m \right\}$ with $0 < \Omega_1 < \Omega_2 < \dots < \Omega_{\bar m}$,
and with $\bar m \leq m$.  The matrix $\cA$ can then be written as
\begin{equation}
  \cA =
  \left(
  \begin{array}{cccc}
    \Omega_1\cJ_1 & & 0 &0  \\
    &\ddots & & \vdots \\
    0&&\Omega_{\bar m}\cJ_{\bar m}& 0 \\
    0 & \cdots &0 & 1
  \end{array}
  \right)
  \begin{array}{l}
    \mbox{ with }
    \cJ_k^2 = -\identity_{F_k},\\
    \mbox{ and }
    F_k \stackrel{\rm def}= N(\cA^2 - \Omega_k^2\identity).
  \end{array}
  \label{eq:rotation-normal-form-2}
\end{equation}
The subspaces $F_k$ are invariant under $\cA$.  Restricted to each $F_k$ the
matrix $\cA$ satisfies $\cA^2 = -\Omega_k^2\identity$, so that
\[
\cJ_k \stackrel{\rm def}=\Omega_k^{-1}\cA
\]
defines a complex structure on $F_k$.

\subsection{An almost Lyapunov function for the dilating-rotating
  soliton flow}

The dilating-rotating soliton equation defines a flow on the unit tangent bundle
$\R^n\times \Sph^{n-1}$ of $\R^n$.  When the rotation matrix $\cA$ vanishes, one
can use Huisken's monotonicity formula for Mean Curvature Flow
\cite{Huisken90,shrinkingdonuts} to show that this flow is exactly the geodesic
flow of the metric
\[
(ds)^2 = e^{\frac\alpha2\|\vC\|^2} \|d\vC\|^2
\]
on $\R^n$.  For translating solitons one can use Ilmanen's monotonicity
formula \cite{1994TomEllipticRegularization} instead to conclude that that
self translating solitons are geodesics for the metric
\[
(ds)^2 = e^{\langle\vv, \vC\rangle} \|d\vC\|^2.
\]
Since geodesic flows are Hamiltonian systems one expects a high
degree of recurrence in the flow, and, indeed, all dilating, non-rotating
solitons are either circles or Abresch-Langer curves.

However, as we are about to show, when $\cA\neq0$ the flow becomes dissipative
and most trajectories accumulate on a small number of circles or Abresch-Langer
curves.  Here we will exhibit an ``almost Lyapunov function'' for the flow
(\ref{eq:soliton-profile}) with $\cA\neq0$ and $\vv=0$.  Recall that a Lyapunov
function for a flow is a function which is non-decreasing along all orbits of
the flow, and is in fact strictly increasing except at fixed points of the flow.
A flow that has a Lyapunov function cannot have periodic orbits, and therefore
the soliton equation cannot have a true Lyapunov function, since the
dilating-rotating circles from \S\ref{sec:trivial-rot-dilators} are periodic
orbits for the flow.  The following lemma presents us with a function that is
non-decreasing along orbits of the soliton flow, and that is constant only on
the trivial dilating-rotating solitons from \S\ref{sec:trivial-rot-dilators}.

\begin{lemma}\label{lem:dVdtau}
  If $\cA\neq0$, then along any solution $\bigl(\vC(s), \vT(s)\bigr)$ of the
  soliton flow the quantity
  \[
  V(\vC, \vT ) = e^{\frac\alpha2\|\vC\|^2 + \langle\vv, \vC\rangle}
  \bigl\langle\vT, \cA\vC\bigr\rangle
  \]
  satisfies
  \begin{equation}
    \frac{dV} {d s}
    = e^{\frac\alpha 2\|\vC\|^2 + \langle\vv, \vC\rangle}
    \|\pro_\vT(\cA\vC)\|^2.
    \label{eq:dVdtau}
  \end{equation}
  In particular one has $\frac{dV} {d s} \geq 0$ with equality only if $\cA\vC$
  is a multiple of $\vT $.
\end{lemma}

\begin{proof} Differentiating $V$, we get
  \begin{equation}
    e^{-\frac\alpha2\|\vC\|^2 - \langle\vv, \vC\rangle}\frac{dV} {d s}
    =\langle\alpha\vC+\vv,\vT \rangle
    \bigl\langle\vT , \cA\vC\bigr\rangle
    + \frac{d} {d s}\bigl\langle\vT , \cA\vC\bigr\rangle
    \label{eq:lyapunov-proof-1}
  \end{equation}
  The last term can be expanded using antisymmetry of $\cA$ and the equation for
  $\vT_s$,
  \begin{eqnarray*}
    \frac{d} {d s}\bigl\langle\vT , \cA\vC\bigr\rangle
    &=&\langle\vT_{s}, \cA\vC\rangle + \langle\vT , \cA\vT \rangle \\
    &=&\langle\vT_{s}, \cA\vC\rangle \\
    &=&\bigl\langle \pro_\vT(\alpha\vC + \cA\vC+\vv), \cA\vC
    \bigr\rangle\\
    &=& \bigl\langle\pro_\vT(\alpha\vC+\vv), \cA\vC\bigr\rangle +
    \bigl\langle\pro_\vT(\cA\vC), \cA\vC\bigr\rangle \\
    &=& \bigl\langle\pro_\vT(\alpha\vC+\vv), \cA\vC\bigr\rangle +
    \|\pro_\vT(\cA\vC)\|^2.
  \end{eqnarray*}
  Substituting this in (\ref{eq:lyapunov-proof-1}) we get
  \begin{eqnarray*}
    \fl \qquad e^{-\frac\alpha2\|\vC\|^2 - \langle\vv, \vC\rangle}\frac{dV} {d s}
    &=&\langle\alpha\vC+\vv,\vT \rangle \bigl\langle\vT , \cA\vC\bigr\rangle +
    \bigl\langle\pro_\vT(\alpha\vC+\vv), \cA\vC\bigr\rangle
    + \|\pro_\vT(\cA\vC)\|^2\\
    &=& \Bigl\langle \bigl\langle\alpha\vC+\vv, \vT\bigr\rangle\vT +
    \pro_\vT(\alpha\vC+\vv), \cA\vC \Bigr\rangle + \|\pro_\vT(\cA\vC)\|^2.
  \end{eqnarray*}
  By definition of $\pro_\vT$ one has $\vw = \pro_\vT(\vw) + \langle\vw,
  \vT\rangle\vT$ for any vector $\vw$.  Therefore
  \[
  e^{-\frac\alpha2\|\vC\|^2 - \langle\vv, \vC\rangle}\frac{dV} {d s} =
  \bigl\langle \alpha\vC+\vv, \cA\vC \bigr\rangle + \|\pro_\vT(\cA\vC)\|^2.
  \]
  The first term on the right vanishes.  Indeed, antisymmetry of $\cA$ implies
  that $\langle\vC, \cA\vC\rangle = 0$, while $\cA\vv=0$ also implies
  $\langle\vv, \cA\vC\rangle = - \langle\cA\vv, \vC\rangle =0$.  This implies
  (\ref{eq:dVdtau}).
\end{proof}

\subsection{The trivial rotating-dilating solitons}
\label{sec:trivial-rot-dilators}
Under Curve Shortening any circle centered at the origin remains a
circle with a varying radius.  It therefore generates a dilating
soliton, and if the circle happens to be invariant under the rotations
$e^{t\cA}$ then it automatically defines a dilating-rotating soliton
which satisfies (\ref{eq:soliton-profile}).  A circle can only remain
invariant under the rotations $e^{t\cA}$ if it is contained in one of
the eigenspaces $F_k$.  Moreover the circle must lie in a complex line
in $F_k$ relative to the complex structure $\cJ_k = \Omega_k^{-1}\cA$. 

If the null space of $\cA$ contains a plane, then any Abresch-Langer
curve in this plane defines a dilating soliton which also rotates
according to $\cA$ (by not rotating at all: $e^{t\cA}$ is the identity on
the null space of $\cA$.)  

\begin{lemma}
  \label{lem:critical-solitons}
  Let $\alpha<0$, and let $\vC$ be a rotating dilating soliton on
  which $V(\vC, \vT)$ is constant.  Then $\vC$ is
  \begin{itemize}
  \item a circle with radius $1/\sqrt{-\alpha}$ in a complex line in one
    of the eigenspaces $F_k$, or
  \item an Abresch-Langer curve in the null space of $\cA$ (or a straight
    line through the origin in $N(\cA)$, which can be considered a
    ``degenerate Abresch Langer curve.'')
  \end{itemize}
  The functional $V$ vanishes on the Abresch-Langer solitons.

  For the circles in $F_k$ one has
  \begin{equation}
    V = \pm \frac{\Omega_k}{\sqrt{-e\alpha}},
    \label{eq:V-value-on-circle}
  \end{equation}
  with the sign depending on their orientation.
\end{lemma}
\begin{proof}
  Since $V$ is constant on $\vC$ we have $\|\cA\vC\|^2 = \bigl\langle\vT,
  \cA\vC\bigr\rangle^2$ and hence $\cA\vC = \gamma(s)\vT$ for some smooth function
  $\gamma$.  Keeping in mind that $\vT$ is a unit vector, so that $\vT_s\perp\vT$,
  we conclude from
  \[
  \frac{d\gamma}{ds} = \bigl\langle\cA\vC, \vT \bigr\rangle_s =
  \underbrace{\bigl\langle\cA\vT, \vT\bigr\rangle}_{=0} + \bigl\langle\cA\vC,
  \vT_{s}\bigr\rangle =\gamma \bigl\langle\vT, \vT_s\bigr\rangle = 0
  \]
  that $\gamma$ must in fact be a constant.

  If $\gamma=0$, then $\cA\vC=0$.  The soliton lies in the kernel of $\cA$, and
  must therefore be an Abresch-Langer curve.  Clearly we have
  $V=0$ on any curve in $N(\cA)$.

  If $\gamma\neq0$, then we have
  \begin{equation}
    \frac{d\vC}{ds} = \frac{1}{\gamma} \; \cA\vC.
    \label{eq:vCs-when-V-is-const}
  \end{equation}
  On the other hand, by taking the inner product with $\vT$ on both sides of the
  soliton equation (\ref{eq:soliton-profile}) we get
  \[
  0 = \bigl\langle\vT, \cA\vC\bigr\rangle + \lambda(s) \|\vT\|^2 = \gamma +
  \lambda(s),
  \]
  where we have used $\vC_s = \gamma^{-1}\cA\vC \implies \vC_s\perp\vC$.  It follows that
  $\lambda = -\gamma$.  Using $\cA\vC = \gamma\vT$ once more we can rewrite the
  soliton equation as
  \[
  \vC_{ss} = \bigl(\alpha +\cA\bigr)\vC + \lambda\vT = \alpha\vC.
  \]
  Differentiating (\ref{eq:vCs-when-V-is-const}) we also get
  \[
  \vC_{ss} = \frac{1}{\gamma}\cA\vC_s = \frac{1}{\gamma^2}\cA^2\vC.
  \]
  Combining these last two equations we see
  \[
  \cA^2\vC = \gamma^2\alpha \vC
  \]
  so that $\vC$ must be an eigenvector of $\cA^2$ with eigenvalue
  $\gamma^2\alpha$.  Thus for some $k$ we have $\vC\in F_k$ and $\gamma^2\alpha
  = - \Omega_k^2$, i.e.
  \[
  \gamma = \pm \frac{\Omega_k}{\sqrt{-\alpha}} .
  \]
  By equation~(\ref{eq:vCs-when-V-is-const}) we have
  \[
  \vC(s) = e^{\pm \sqrt{-\alpha} s \cJ_k} \vC(0),
  \]
  which shows that $\vC$ is a circle in the complex line in $F_k$ that contains
  $\vC(0)$.  The radius of the circle is $\|\vC(0)\|$, and is determined by the
  condition $\|\vC_s\| = 1$:
  \[
  \|\vC_s\|
  = \left\|\pm \sqrt{-\alpha}  \cJ_k e^{\pm \sqrt{-\alpha} s \cJ_k}
  \vC(0)\right\|
  =\sqrt{-\alpha} \|\vC(0)\| =1 \implies \|\vC(0)\| = \frac{1} {\sqrt{-\alpha}}.
  \]
  The value of $V$ on this soliton is
  \[
  V = e^{-\frac\alpha 2 \|\vC\|^2} \bigl\langle\vT, \cA\vC\bigr\rangle
  = \gamma e^{-1/2}
  = \pm \frac{\Omega_k}{\sqrt{-e\alpha}} .
  \]

\end{proof}

\subsection{Compactification of the Soliton Flow}
One obstacle for studying the global behaviour of solitons is that the
phase space $\R^n\times \Sph^{n-1}$ on which the soliton flow
(\ref{eq:soliton-profile-as-system}) is defined is not compact.  To
overcome this we show here that, after reparameterizing the flow, one
can compactify the phase space and extend the flow by adding points at
infinity.

We map $\R^n$ onto the interior of the unit ball via $\vC\mapsto\vP$,
\[
\vP = \frac{\vC} {\sqrt{1+\|\vC\|^2}}.
\]
The inverse of this map is given by
\[
\vC = \frac{\vP} {\sqrt{1-\|\vP\|^2}}.
\]
A short computation shows that if $(\vC, \vT)$ satisfy
(\ref{eq:soliton-profile-as-system}), then
\[
\vP_s = \frac{\vC_s} {\sqrt{1+\|\vC\|^2}}
- \frac{\langle\vC,\vC_s\rangle\vC} {\bigl(1+\|\vC\|^2\bigr)^{3/2}}.
\]
This implies
\[
\sqrt{1-\|\vP\|^2}\; \vP_s
= \bigl(1-\|\vP\|^2\bigr) \bigl(\vT - \langle\vP, \vT\rangle \vP\bigr).
\]
The evolution of $\vT$ can be written as
\[
\sqrt{1-\|\vP\|^2}\; \vT_s
= \pro_\vT \bigl((\alpha + \cA)\sqrt{1-\|\vP\|^2} \vC\bigr)
= \pro_\vT \bigl((\alpha + \cA)\vP\bigr).
\]
Hence, if we introduce a new parameter $\varsigma$ along the soliton
that is related to arc length via
\[
d\varsigma = \frac{ds} {\sqrt{1-\|\vP\|^2}} = \sqrt{1+\|\vC\|^2} \; ds,
\]
then the soliton system (\ref{eq:soliton-profile-as-system}) is
equivalent with the following system for $(\vP, \vT)$
\begin{eqnarray}
  \vP_\varsigma & =&  \bigl(1-\|\vP\|^2\bigr) \bigl(\vT - \langle\vP,
  \vT\rangle \vP\bigr) \\
  \vT_\varsigma & =& \pro_\vT \bigl((\alpha + \cA)\vP\bigr).
\end{eqnarray}%
\label{eq:soliton-profile-compactified}%
This system is obviously well-defined at all $(\vP, \vT) \in
\R^n\times \Sph^{n-1}$.  The original phase space, which was described
in terms of $\vC$ and $\vT$, is embedded in the $(\vP, \vT)$ phase
space as $\{(\vP, \vT) : \|\vP\|<1, \|\vT\|=1\}$.  The closure of this
set,
\[
\B^n\times \Sph^{n-1} = \{(\vP, \vT) : \|\vP\|\leq 1, \|\vT\|=1\}
\]
is invariant under the extended flow
(\ref{eq:soliton-profile-compactified}), and it is also compact.

The almost Lyapunov function $V$ also extends to the compactified
phase space, at least when $\alpha<0$.  Indeed, direct substitution
shows that
\[
V = e^{\alpha \|\vP\|^2/(1-\|\vP\|^2)} \bigl(1-\|\vP\|^2\bigr)^{-1/2}
\langle\cA\vP, \vT\rangle
\]
for $\|\vP\| < 1$.  When $\alpha<0$ the exponential is the dominant
factor, which causes the whole expression to extend to a $C^\infty$
function on $\B^n\times \Sph^{n-1}$.  One has $V=0$ at ``the points at
infinity,'' i.e.~when $\|\vP\|=1$.

Since $V$ is constant on the added points at infinity, one easily
verifies that the continuous extension of $V$ to $\B^n\times
\Sph^{n-1}$ still is an almost Lyapunov function for the flow
(\ref{eq:soliton-profile-compactified}):

\begin{lemma}
  \label{lem:V-critical-extended}
  Let $\vX(\varsigma) = (\vP(\varsigma), \vT(\varsigma))$ be an orbit
  in $\B^n\times \Sph^{n-1}$ of the compactified soliton flow
  (\ref{eq:soliton-profile-compactified}).  Then $\varsigma \mapsto
  V(X(\varsigma))$ is either strictly decreasing or constant along the
  entire orbit.

  If $V(\vX(\varsigma))$ is constant then $\vX$ corresponds to one of
  the trivial solitons from Lemma~\ref{lem:critical-solitons}, or
  the entire orbit $\vX(\varsigma)$ consists of ``points at infinity,''
  i.e.~$\|\vP(\varsigma)\| = 1$ for all $\varsigma\in\R$.
\end{lemma}

\begin{figure}[t]
  \centering
  \includegraphics[width=\textwidth]{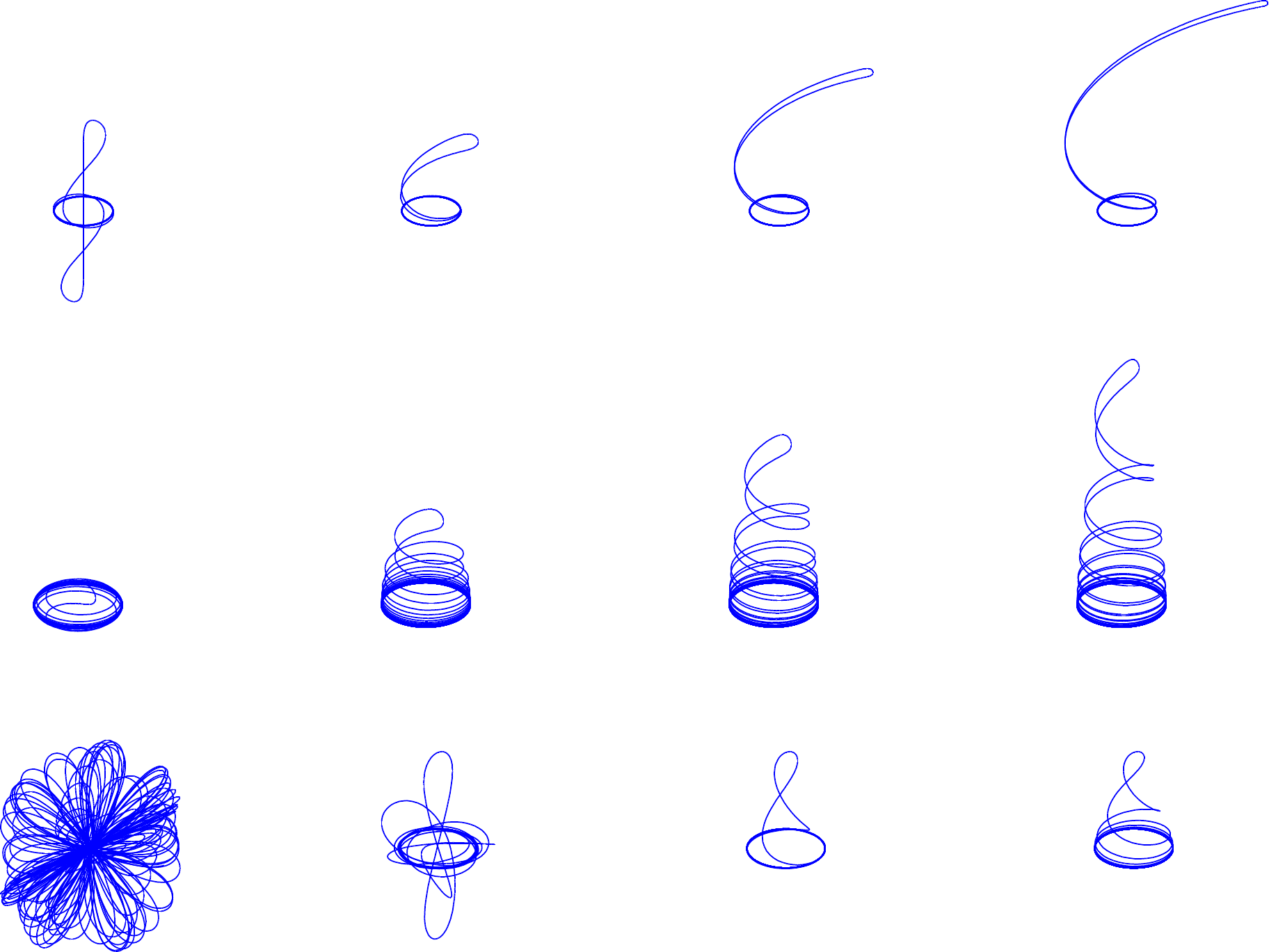}
  \caption{\textbf{Rotating and shrinking solitons limit to a circle.}
    Shown are solitons that rotate around the $z$-axis and shrink
    toward the origin simultaneously.  A circle in the $xy$-plane with
    radius $\sqrt{2}$ is one such soliton.  Both ends of any generic
    soliton of this type are asymptotic to the special circle.}
\end{figure}

\begin{figure}[t]
  \begin{center}
    \includegraphics[width=0.8\textwidth]{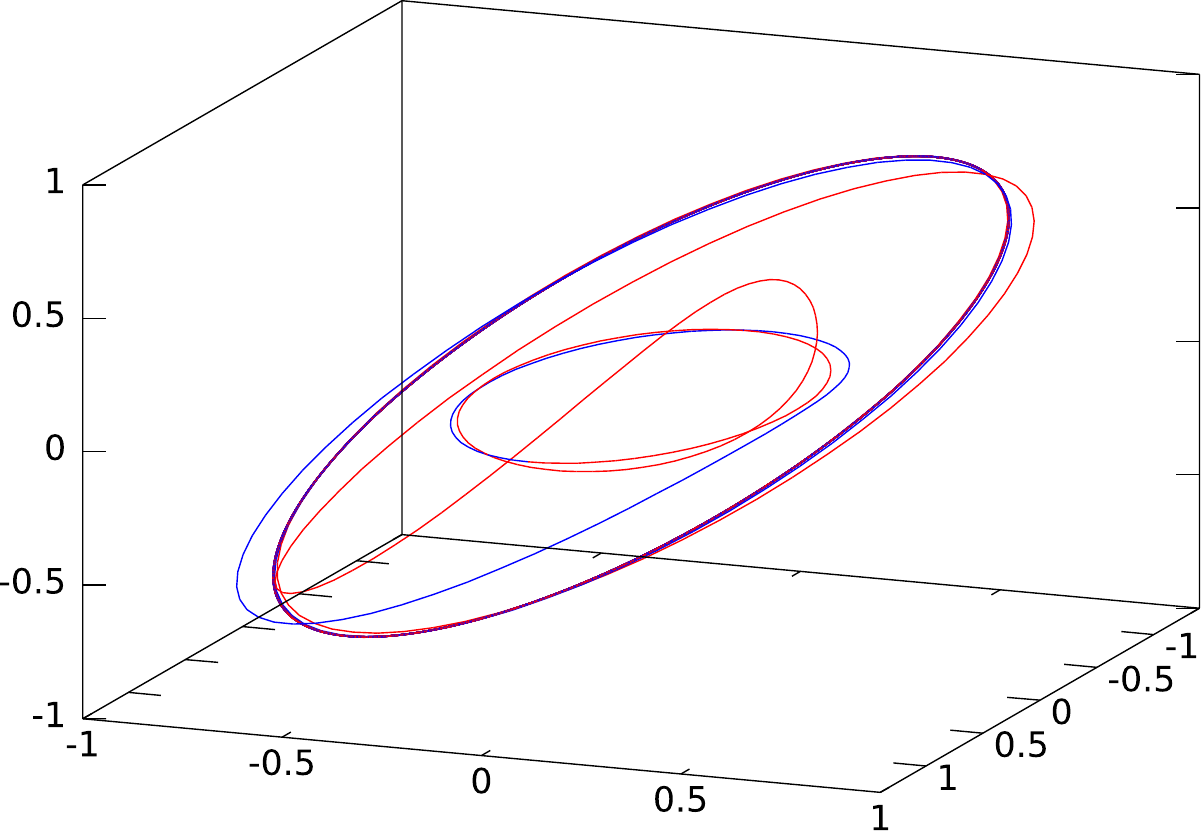}
  \end{center}
  \caption{A three dimensional projection of a shrinking rotating
    soliton in $\R^4$.  Here $\alpha=-1$, $\Omega_1 = \frac12$,
    $\Omega_2=2$.  One end of the soliton is colored blue and the
    other red, so one can see that both ends are asymptotic to the
    same circle, but with opposite orientations.  The limiting circle
    is the unit circle in the eigenspace $F_2$ (which appears as the
    larger ellipse in this projection), while the middle part of the
    soliton remains close to the unit circle in the eigenspace $F_1$.
  }
  \label{fig:4D-shrinker-rotator}
\end{figure}
\subsection{The ends of rotating-dilating solitons}
Let $\vC:\R\to\R^n$ be a solution of (\ref{eq:soliton-profile}) with
$\alpha < 0$, and write $\vX(s) = (\vP(s), \vT(s)) \in \B^n \times
\Sph^{n-1}$ for the corresponding solution to the system
(\ref{eq:soliton-profile-compactified}).  By definition, the
$\omega$-limit set of the orbit $\{\vX(s) : s\in\R\} $ consists of
all possible limits $ \vX_* = \lim_{n\to\infty} \vX(s_n)$ one can
obtain by choosing arbitrary sequences $s_n\nearrow\infty$, i.e.
\begin{eqnarray*}
  \omega\bigl(\vX\bigr) &\stackrel{\rm def}= &
  \Bigl\{ \vX_* \in \B^n \times \Sph^{n-1} : (\exists s_n \nearrow \infty)
  \lim_{n\to\infty} \vX_* = \vX(s_n) \Bigr\} \\
  &=&
  \bigcap_{s>0} \overline{
  \bigl\{\vX(s') : s'\geq s\bigr\}
  }.
\end{eqnarray*}
The second description shows that $\omega(\vX)$ is always a
compact connected subset of $\B^n \times \Sph^{n-1}$, which is invariant
under the flow (\ref{eq:soliton-profile-compactified}).  

The $\alpha$-limit set $\alpha(\vX)$ of the orbit $\{\vX(s):s\in\R\}$ is defined
similarly as the set of all possible limits of $\vX(s_n)$ for sequences
$s_n\searrow -\infty$.

The compactification of the soliton flow, and the almost Lyapunov
function now give us the following description of the global behaviour
of rotating-dilating solitons.

Let $\vC:\R\to\R^n$ be a rotating-dilating soliton, i.e.~a solution
of (\ref{eq:soliton-profile}) with $\alpha<0$, and let $\vX(\varsigma)
= \bigl(\vP(\varsigma), \vT(\varsigma)\bigr)$ be the corresponding
orbit of (\ref{eq:soliton-profile-compactified}).  Then, since
$V$ is monotone along the soliton, and since $V$ is obviously bounded,
the limits
\[
V_\pm \stackrel{\rm def}=
\lim_{s\to\pm\infty} V(\vC(s), \vC_s(s)) 
\]
must exist.  One has $V_-\leq V_+$.  Moreover, $V$ is constant on both
$\omega(\vX)$ and $\alpha(\vX)$ with
\[
V|_{\alpha(\vX)} = V_-,
\qquad
V|_{\omega(\vX)} = V_+.
\]
By Lemma~\ref{lem:V-critical-extended} both $\alpha$ and $\omega$ limit
sets must consist of trivial rotating-dilating solitons, Abresch
Langer curves in $N(\cA)$, or points at infinity.

We consider the possibilities for the $\omega$-limit set (the $\alpha$-limits
are analogous.)

\begin{lemma}
  If $V_+\neq0$ then $V_+ = \pm \Omega_k/\sqrt{-e\alpha}$ for some $k$, and
  $\omega(\vX)$ consists of complex circles with radius $(-\alpha)^{-1/2}$ in
  the subspace $F_k$, as found in Lemma~\ref{lem:critical-solitons}.  

  In the special (but generic) case that all eigenvalues $i\Omega_k$ of
  $\cA$ are simple, the subspaces $F_k$ are two-dimensional, and they
  contain exactly one soliton-circle: when this happens $\omega(\vX)$
  contains that circle, and the soliton $\vC$ is asymptotic to the circle.
\end{lemma}

The case $V_+ = 0$ is more complicated.  If $V_+=0$ then $\omega(\vX)$ could
contain points at infinity, compact Abresch-Langer curves in $N(\cA)$, or straight
lines through the origin in $N(\cA)$.

\begin{lemma}
  Assume $V_+ = 0$.
  If $\omega(\vX)$ is not contained in $\overline{N(\cA)}\times \Sph^{n-1}$ then
  $\vC(s) \in \cR_-(K)$ for some large enough $K$ and all large enough
  $s$.  In this case the soliton $\vC$ is asymptotic to a log spiral as in
  (\ref{eq:asymptotic-log-spiral-alpha-neg}).
  In particular $\omega(\vX)$ is disjoint from $N(\cA)$.
  \label{lem:spiral-at-infinity}
\end{lemma}
\begin{proof}
  Since $V_+=0$ we know that $\omega(\vX)$ is a subset of
  $(\partial\B^{n} \cup N(\cA)) \times \Sph^{n-1}$.  If there is a
  $(\vP_0,\vT_0) \in \omega(\vX) \cap \partial\B^n$ with $\cA\vP_0\neq 0$,
  then we can choose a sequence $s_n\to\infty$ for which $\|\vC(s_n)\| \to
  \infty$, and $\vC(s_n)/\|\vC(s_n)\| \to \vP_0$.

  In view of $ \va(s_n) = (\alpha+\cA) \vC(s_n) $ we have
  \[
  \lim _{n\to\infty} \|\va(s_n)\| = \infty.
  \]
  If $\nu(s_n)>K$ and if $(\vC(s_n'), \vT(s_n')) \not\in \cR_-(K) \cup
  \cR_+(K)$ then Lemma~\ref{lem:GR-here} implies that for some $s_n' = s_n +
  o(1)$ one has $(\vC(s_n'), \vT(s_n')) \in \cR_+(K)$.
  So after changing the $s_n$ slightly we may assume that $(\vC(s_n),
  \vT(s_n)) \in \cR_-(K) \cup \cR_+(K)$.
  If $(\vC(s_n), \vT(s_n)) \in \cR_+(K)$, then at $s=s_n$ we have
  \begin{eqnarray*}
    \langle\cA\vC, \vT\rangle
    &= \langle\cA\vC, \ah\rangle + \langle\cA\vC, \vT-\ah\rangle \\
    &= \frac{\langle\cA\vC, \alpha\vC + \cA\vC\rangle}{\|(\alpha+\cA)\vC\|}
    + \langle\cA\vC, \vT-\ah\rangle \\
    &= \frac{\|\cA\vC\|^2}{\|(\alpha+\cA)\vC\|}
    + \langle\cA\vC, \vT-\ah\rangle \\
    &\geq \|\cA\vC\|
    \Bigl\{
    \frac{\|\cA\vC\|}{\|(\alpha+\cA)\vC\|} -\|\vT-\ah\|
    \Bigr\}
  \end{eqnarray*}
  On the other hand $V(\vC(s), \vT(s))$ is a nondecreasing function which
  converges to zero, so $V(\vC(s), \vT(s))\leq 0$, and thus
  $\langle\cA\vC, \vT\rangle \leq 0$ for all $s$.  Hence we have shown that
  at $s=s_n$
  \[
  \frac{\|\cA\vC\|}{\|\vC\|}
  \leq \frac{\|(\alpha+\cA)\vC\|}{\|\vC\|}\; \|\vT-\ah\|
  \leq \|\alpha+\cA\| \;  \|\vT-\ah\|.
  \]
  Since we are assuming that $(\vC(s_n), \vT(s_n)) \in \cR_+(K)$ we know that
  $\|\vT(s_n) - \ah(s_n)\| \to 0$, and therefore we end up with 
  \[
  \|\cA\vP_0\| = \left\|\lim_{n\to\infty} \cA\frac{\vC(s_n)}{\|\vC(s_n)\|}
  \right\| = 0.
  \]
  This is impossible since we have assumed $\vP_0 \not \in N(\cA)$.  The
  contradiction shows that we must have $(\vC(s_n), \vT(s_n))\in\cR_-(K)$ for
  all $n$.  Moreover, if $(\vC(s_n), \vT(s_n))\in\cR_-(K)$ and if $(\vC(s),
  \vT(s))$ were to leave $\cR_-(K)$ for some $s\geq s_n$ then it follows
  from Lemma~\ref{lem:GR-here} that $(\vC(s), \vT(s))$ would end up in
  $cR_+(K)$.  The previous arguments show that this again contradicts
  $V_+\leq0$.  We can therefore conclude that $(\vC(s), \vT(s))$ must remain in
  $\cR_-(K)$ for all $s\geq s_n$.  
  Lemma~\ref{lem:alpha-pos-R-invariant} implies that the soliton is
  asymptotic to a generalized logarithmic spiral.
\end{proof}

If $\omega(\vX)$ is contained in $N(\cA)\times \Sph^{n-1}$ then all we
can say now is that $\omega(\vX)$ consists of lines and Abresch-Langer
curves (and $\omega(\vX)$ must of course be compact and connected.)  It would be
natural to conjecture that if $\omega(\vX)$ contains an Abresch-Langer
curve, then that curve is all of $\omega(\vX)$, and the soliton converges
to that Abresch-Langer curve.  It is also natural to conjecture that
$\omega(\vX)$ can never contain a straight line, since the soliton would
have to march up and down this line infinitely often.
However, preliminary computations involving matched asymptotic expansions
seem to show that even in the 3 dimensional case a shrinking-rotating
soliton exists for which $\alpha(\vX)$ is a circle in the $xy$-plane, while
$\omega(\vX)$ consists of the entire $z$-axis.  The construction and
analysis of this solution will be the subject of a future paper.

\bibliographystyle{amsplain}
\bibliography{SpaceCurveReferences}

\end{document}